\documentclass[11pt]{amsart}
\usepackage{graphicx,amsmath,amssymb,latexsym,bbm,mathtools,wasysym}
\usepackage{enumerate}
\usepackage{fullpage}
\usepackage{subfigure}
\usepackage{todonotes}

\makeatletter
\@namedef{subjclassname@2020}{%
  \textup{2020} Mathematics Subject Classification}
\makeatother

\newtheorem{theorem}{Theorem}[section]
\newtheorem{corollary}[theorem]{Corollary}
\newtheorem{lemma}[theorem]{Lemma}
\newtheorem{proposition}[theorem]{Proposition}
\theoremstyle{definition}
\newtheorem{definition}[theorem]{Definition}
\theoremstyle{remark}
\newtheorem{remark}[theorem]{Remark}
\theoremstyle{remark}
\newtheorem{example}[theorem]{Example}
\numberwithin{equation}{section}

\newcommand{\abs}[1]{\left\vert#1\right\vert}
\newcommand{\set}[1]{\left\{#1\right\}}
\newcommand{\R}{\mathbb R}
\newcommand{\C}{\mathbb C}
\newcommand{\N}{\mathbb N}

\newcommand{\PP}{\mathbb P}
\newcommand{\UU}{\mathbbm 1}
\newcommand{\BB}{{\mathcal B}}
\newcommand{\CC}{{\mathcal C}}
\newcommand{\II}{{\mathcal I}}
\newcommand{\JJ}{{\mathcal J}}
\newcommand{\FF}{{\mathcal F}}
\newcommand{\LL}{{\mathcal L}}
\newcommand{\PPP}{{\mathcal P}}
\newcommand{\TTT}{{\mathcal T}}
\newcommand{\rd}{\mathrm d}

\DeclareMathOperator*{\esssup}{ess\,sup}
\DeclareMathOperator*{\essinf}{ess\,inf}
\DeclareMathOperator*{\diam}{diam}

\begin{document}
\title{Almost sure asymptotic behaviour of Birkhoff sums for infinite measure-preserving dynamical systems}

\author{Claudio Bonanno}
\address{Dipartimento di Matematica, Universit\`a di Pisa, Largo Bruno Pontecorvo 5, 56127 Pisa, Italy}
\email{claudio.bonanno@unipi.it}

\author{Tanja I. Schindler}
\address{Centro di Ricerca Ennio De Giorgi, Scuola Normale Superiore, Piazza dei Cavalieri 3, 56126 Pisa, Italy}
\email{tanja.schindler@sns.it}

\begin{abstract}
We consider a conservative ergodic measure-preserving transformation $T$ of a $\sigma$-finite measure space $(X,\BB,\mu)$ with $\mu(X)=\infty$. Given an observable $f:X\to \R$, we study the almost sure asymptotic behaviour of the Birkhoff sums $S_Nf(x) := \sum_{j=1}^N\, (f\circ T^{j-1})(x)$. In infinite ergodic theory it is well known that the asymptotic behaviour of $S_Nf(x)$ strongly depends on the point $x\in X$, and if $f\in L^1(X,\mu)$, then there exists no real valued sequence $(b(N))$ such that $\lim_{N\to\infty} S_Nf(x)/b(N)=1$ almost surely. 
In this paper we show that for dynamical systems with strong mixing assumptions for the induced map on a finite measure set, there exists a sequence $(\alpha(N))$ and $m\colon X\times \N\to\N$ such that for $f\in L^1(X,\mu)$ we have $\lim_{N\to\infty} S_{N+m(x,N)}f(x)/\alpha(N)=1$ for $\mu$-a.e.\ $x\in X$. Instead in the case $f\not\in L^1(X,\mu)$ we give conditions on the induced observable such that there exists a sequence $(G(N))$ depending on $f$, for which $\lim_{N\to\infty} S_{N}f(x)/G(N)=1$ holds for $\mu$-a.e.\ $x\in X$.
\end{abstract}

\subjclass[2020]{37A40, 37A25, 60F15}
\keywords{Infinite ergodic theory; almost sure limits for Birkhoff sums; $\psi$-mixing; trimmed sums}
\thanks{The authors are partially supported by 
the PRIN Grant 2017S35EHN$\_$004 ``Regular and stochastic behaviour in dynamical systems'' of the Italian Ministry of University 
and Research (MUR), Italy. This research is part of the authors' activity within 
the UMI Group ``DinAmicI'' \texttt{www.dinamici.org} and of the first author's activity within the 
Gruppo Nazionale di Fisica Matematica, INdAM, Italy. 
The second author acknowledges the support of the Centro di Ricerca Matematica
Ennio de Giorgi and of UniCredit Bank R\&D group for financial support through the
‘Dynamics and Information Theory Institute’ at the Scuola Normale Superiore.}

\maketitle
\section{Introduction}

Let $T:X\to X$ be a conservative ergodic measure-preserving transformation of a $\sigma$-finite measure space $(X,\BB,\mu)$ and consider an observable $f$, that is a measurable function $f:X\to \R$. The almost sure asymptotic behaviour of the \emph{Birkhoff sums of $f$}
\[
S_Nf(x):= \sum_{j=1}^N\, (f\circ T^{j-1})(x)
\]
is the main result of the famous Birkhoff's Ergodic Theorem. It states that if $\mu$ is a probability measure and $f\in L^1(X,\mu)$, then $S_Nf(x)$ is asymptotic to $N\int_X f\,d\mu$ as $N\to \infty$ for $\mu$-a.e.\ $x\in X$. On the other hand, if $f\not\in L^1(X,\mu)$ the almost sure asymptotic behaviour of $S_Nf$ is not described by any sequence $(b(N))$. In particular, letting $f:X \to \R_{\ge 0}$, one can apply \cite[Cor. 2.3.4]{aa-book} to show that given a sequence $(b(N))$ of positive numbers with $b(N)/N \to \infty$, for $\mu$-almost every (a.e.) $x\in X$ either $\limsup_N S_Nf(x)/b(N) = \infty$ or $\liminf_N S_Nf(x)/b(N) = 0$. 

The situation is different if the measure $\mu$ is $\sigma$-finite but infinite, that is $\mu(X)=\infty$. In this case if $f\in L^1(X,\mu)$ then $S_Nf(x) = o(N)$ for $\mu$-a.e.\ $x\in X$, but again the almost sure asymptotic behaviour of $S_Nf$ is not described by any sequence $(b(N))$ slower than $N$. It is the content of \cite[Thm. 2.4.2]{aa-book} that given a sequence $(b(N))$ of positive numbers, even with $b(N)/N \to 0$, for $\mu$-a.e.\ $x\in X$ either $\liminf_N S_Nf(x)/b(N) = 0$ for all $f\in L^1(X,\mu)$, or there exists a sequence $(N_k)$ such that $\lim_k S_{N_k}f(x)/b(N_k) = \infty$ for all $f\in L^1(X,\mu)$. For infinite measure-preserving dynamical systems, an information about the almost sure asymptotic behaviour of the Birkhoff sums of summable observables is given in Hopf's Ratio Ergodic Theorem (see \cite[Thm 2.2.5]{aa-book}). It is shown that for $f,g\in L^1(X,\mu)$ with $g\ge 0$ and $\int_X g\,d\mu >0$, the ratio $S_Nf(x)/S_Ng(x)$ converges to $\int_X f\,d\mu/\int_X g\,d\mu$ as $N\to \infty$ for $\mu$-a.e.\ $x\in X$. That is, the almost sure asymptotic behaviour of $S_Nf$ is the same for all $f\in L^1(X,\mu)$, and its variability which leads to the negative results we have recalled above, depends on the initial condition $x\in X$. As far as we know, the almost sure asymptotic behaviour of Birkhoff sums has not been studied for non-summable observables in the context of infinite ergodic theory, with the sole exception of \cite{lenci}.

In this paper we restrict our attention to the case that $\mu$ is infinite and study the almost sure asymptotic behaviour of Birkhoff sums for both summable and non-summable observables with respect to sequences. We consider dynamical systems with strong mixing assumptions, namely \emph{$\psi$-mixing} (see Definition \ref{def-psi-mixing}), for the induced map on a finite measure subset $E\subset X$ (see Section \ref{setting} for the detailed assumptions), and prove two main results. First we prove that for $f\in L^1(X,\mu)$ there exist sequences of positive real numbers $(\alpha(N))$ and $(m(N,E,x))$, the second sequence depending additionally on $x$, such that $S_{N+m(N,E,x)}f(x)$ is asymptotic to $\alpha(N) \int_X f\,d\mu$ as $N\to \infty$ for $\mu$-a.e.\ $x\in X$. That is to obtain the same asymptotic behaviour for $\mu$-a.e.\ $x\in X$ we have to change the number of terms to consider in the Birkhoff sums (see Theorem \ref{th-elle1}). Then we prove that for a non-summable observable $f:X\to \R_{\ge 0}$, under suitable assumptions on its induced version $f^E$ (see \eqref{induced-obs}) it is possible to find a sequence $(G(N))$ which depends on $f$ such that $S_Nf(x) \sim G(N)$ as $N\to \infty$ for $\mu$-a.e.\ $x\in X$ (see Theorem \ref{th-notelle1}). The main difference is that in this non-summable case the properties of $f$ play an important role.

The proofs rely on the method of \emph{trimmed sums} (see \eqref{trimmed-birk-sum}), that is sums from which a number of largest entries is deleted. The method to trim sums of independent random variables to prove (pointwise and distributional) limit theorems which fail to hold for the untrimmed sum is long established and most results were developed in the 80th and 90th of the last century. One generally differentiates between different strengths of trimming, according to the number $r$ of largest entries deleted. In the light trimming case the number of deleted entries $r$ is constant for all $n$, while in the intermediate trimming case the number of deleted entries $r(n)$ depends on $n$ and tends to infinity, but with $r(n)=o(n)$.

The first result using trimming in the dynamical systems context was given by Diamond and Vaaler \cite{diamond_estimates_1986} providing a strong law of large numbers for the continued fraction digits under the use of light trimming. More precisely, if $(a_n(x))$ are the coefficients of the continued fraction expansion of a point $x\in [0,1]$, it is proved in \cite{diamond_estimates_1986} that $\lim_{n\to\infty}(\sum_{k=1}^n a_k(x)- \max_{1\leq k\leq n} a_k(x))/(n\log n)=\log 2$ holds for Lebesgue-a.e.\ $x\in [0,1]$. Hence in this case it is enough to trim the sum of the coefficients $(a_n(x))$ by deleting only the largest entry, so $r=1$. These results were generalised by Aaronson and Nakada \cite{aar-nakada} giving strong laws of large numbers under light trimming for sufficiently fast $\psi$-mixing random variables, thus giving the analog to some statement for independent random variables from \cite{kestenmaller}. Haynes further quantified in \cite{haynes} the results by Aaronson and Nakada by giving precise error terms. 

Kesseb\"ohmer and the second author of this paper proved strong laws of large numbers under intermediate trimming using a spectral gap property of the transfer operator, see \cite{kessboehmer_strong_2019} and \cite{kesseboehmer_intermediately_2019} for an application of these results to subshifts of finite type and \cite{haeusler_laws_1987, haeusler_nonstandard_1993, kesseboehmer_strongiid_2019} for these and further reaching results in the independent case. See further \cite{kesseboehmer_mean_2019} for a convergence in mean result concerning the same intermediately trimmed sums. 

The paper is organised as follows. In Section \ref{setting} we describe the setting we consider and state the main results, Theorems \ref{th-elle1} and \ref{th-notelle1}, whose proofs are collected in Section \ref{proofs}. In Section \ref{sec-examples} we describe two classes of dynamical systems to which our results are applicable, and discuss the assumptions on the observables for Theorem \ref{th-notelle1}. We also study in more detail our results for the Farey map, a well-known dynamical system on the interval, also in relation to the results in \cite{lenci} which are shown to be a particular case of ours for this system. Finally, two appendices complete the paper. In particular Appendix \ref{sec:svf} contains definitions and results on slowly varying functions, an important class of functions used in the statement of Theorem \ref{th-notelle1}. 

We conclude with some comments and a discussion on possible directions of future research. As we have stated above, our main results hold for dynamical systems for which there exists a finite measure set $E$ on which the induced map satisfies a $\psi$-mixing condition. It would be interesting to study if this $\psi$-mixing condition could be relaxed. This seems to be highly unlikely in general. Aaronson and Nakada already gave in \cite{aar-nakada} an example of a mixing though not $\psi$-mixing system which does not fulfil a strong law of large numbers after trimming, even though a sufficiently fast $\psi$-mixing system with the same distribution functions would fulfil a trimmed strong law. Haynes gave another simple example with even stronger mixing properties which does not fulfil a lightly trimmed strong law, see \cite[Thm. 4]{haynes}. However, this particular system still fulfils an intermediately trimmed strong law, see \cite{schindler_observables_2018}. Thus, there might be possibilities to generalise Theorems \ref{th-elle1} and \ref{th-notelle1} in this way. In particular, in this paper we have considered assumptions on the system $(X,T,\mu)$ for which light trimming is enough to get the results, actually it has been enough to trim the sums by deleting only the largest entry. As shown in Section \ref{subsec:interval-maps} for maps of the interval with indifferent fixed points, this is related with the order with which the derivative of the map converges to 1 in the indifferent fixed points. Using light trimming we have been able to consider only the so-called ``barely infinite'' situation. It is possible that we can obtain similar results for systems with a different ``order of infinity'' by using intermediate trimming.

\section{The setting and main results} \label{setting}
Let $T:X\to X$ be a conservative ergodic measure-preserving transformation of the measure space $(X,\BB,\mu)$ with $\mu$ a $\sigma$-finite measure with $\mu(X)=\infty$. Let $E\in \BB$ with $\mu(E)=1$, and denote by $\varphi_{_E}$ the \emph{first return time}
\[
\varphi_{_E} : E \to \N\, ,\qquad \varphi_{_E}(x):= \min \set{k\ge 1\, :\, T^k(x)\in E}.
\]  
The first return time is finite for $\mu$-a.e.\ $x\in E$, so we define $T_{_E}(x) := T^{\varphi_{_E}(x)}(x)$ to be the \emph{induced map} $T_{_E}:E \to E$, which is an ergodic measure-preserving transformation of the probability space $(E,\BB|_E,\mu)$. The first return time generates two families of sets, the level sets
\begin{equation} \label{level-sets}
A_n := \set{x\in E\, : \varphi_{_E}(x)=n}
\end{equation}
which are a measurable partition of $E$, and the super-level sets
\begin{equation}\label{super-level-sets}
A_{>n} := \set{x\in E\, :\, \varphi_{_E}(x) > n} = \bigsqcup_{k>n}\, A_k.
\end{equation}
We also use the notation $A_{\ge n} = A_n \cup A_{>n}$. Applying Kac's Theorem one has
\begin{equation} \label{somma-inf}
\sum_{k\ge 1}\, k\, \mu(A_k) = \sum_{n\ge 0}\, \mu(A_{>n}) = \mu(X) = \infty.
\end{equation}
For this reason the vanishing order of the sequence $( \mu(A_{>n}) )$ is called the \emph{order of infinity} of the system $(X,T,\mu)$. It is known that this order is independent of the finite measure subset $E$ we have fixed.

A useful notion is the \emph{longest excursion out of $E$ beginning in the first $N$-steps} defined for $\mu$-a.e.\ $x\in X$ as
\begin{equation}\label{def-m}
m(N,E,x) := 1+\max \set{ k\ge 1\, :\, \exists\, \ell \in \{1,\dots,N+1\} \text{ s.t. } T^{\ell+j}(x) \not\in E,\, \forall\, j=0,\dots,k-1}.
\end{equation}
If we denote by $R_{_{E,N}}(x)$ the \emph{number of visits to $E$ up to time $N$} along the orbit of a point $x$, that is
\begin{equation} \label{numb-visits}
R_{_{E,N}}(x) := \sum_{k=1}^{N+1}\, (\UU_{_E}\circ T^{k-1})(x)
\end{equation}
it is immediate to verify that for $\mu$-a.e.\ $x\in E$
\begin{equation}\label{relations}
m(N,E,x) = \max\set{(\varphi_{_E}\circ T_{_E}^{k-1})(x)\, :\, k=1,\dots, R_{_{E,N}}(x)}.
\end{equation}
A notion related to $m(N,E,x)$ is that of the \emph{longest excursion out of $E$ seen up to time $N$,} defined for $\mu$-a.e.\ $x\in E$ by
\begin{equation}\label{def-w}
w(N,E,x) := 1+ \max \set{ k\ge 1\, :\, \exists\, \ell \in \{1,\dots,N-k\} \text{ s.t. } T^{\ell+j}(x) \not\in E,\, \forall\, j=0,\dots,k-1}. 
\end{equation}
Notice that $w(N,E,x)$ can correspond to an excursion which does not return to $E$ up to time $N$.

In this paper we are interested in studying the pointwise asymptotic behaviour of the Birkhoff sums
\[
S_N f (x) := \sum_{n=1}^{N}\, (f\circ T^{n-1})(x)
\]
for summable and non-summable observables $f:X\to \R$. For $f\in L^1(X,\mu)$ it is a straightforward consequence of Hopf's Ratio Ergodic Theorem \cite[Thm 2.2.5]{aa-book} that $\mu$-a.s.\ $S_N f(x)=o(N)$. Moreover, the exact asymptotic pointwise behaviour of $S_Nf$ cannot be recovered for all $f\in L^1(X,\mu)$ by changing the normalising sequence due to Aaronson's Ergodic Theorem \cite[Thm 2.42]{aa-book}, which basically states that for any sequence of positive real numbers the growth rate of the Birkhoff sums of a non-negative summable observable will almost surely be either over- or under-estimated infinitely often. In this paper we obtain some results for the pointwise asymptotic behaviour of $S_N f$ under strong mixing assumptions for the induced system $(E,T_{_E},\mu)$.

We now state the definitions and the results we need in the general case of random variables. We refer to \cite{bradley} for more definitions.

\begin{definition}\label{def-psi-mixing}
Let $(Y_n)$ be a sequence of random variables on a probability space $(Y,\PP)$, and let $\FF_h^k$, for $0\le h<k\le \infty$, be the $\sigma$-field generated by $(Y_n)_{h\le n\le k}$. The sequence $(Y_n)$ is \emph{$\psi$-mixing} if
\[
\psi(n) := \sup \left\{ \Big| \frac{\PP(B\cap C)}{\PP(B)\PP(C)} -1 \Big|\, :\, B\in \FF_0^j, \, C\in \FF_{j+n}^\infty,\, \PP(B)>0,\, \PP(C)>0,\, j\in \N \right\}
\]
satisfies $\psi(n)\to 0$ as $n\to \infty$.
\end{definition}

We now recall the notion of \emph{lightly trimmed sums}. Given a sequence $(Y_n)$ of random variables on a probability space $(Y,\PP)$ and a point $y\in Y$, for each $N\in \N$ we choose a permutation $\pi$ of $\{1,2,\dots,N\}$ such that $Y_{\pi(1)}(y) \ge Y_{\pi(2)}(y) \ge \dots \ge Y_{\pi(N)}(y)$. For a given $r\in \N_0$, the lightly trimmed sum of the $(Y_n)$ is defined by
\begin{equation}\label{trimmed-birk-sum}
S_N^r(y) := \sum_{n=r+1}^N\, Y_{\pi(n)}(y) 
\end{equation}
that is the sum of the first $N$ random variables trimmed by the largest $r$ entries.
 
The next result is adapted from \cite{aar-nakada} as shown in Appendix \ref{sec:lemma-aar-nakada}.

\begin{lemma}\label{lemma-aar-nakada}
Let $(Y_n)$ be a sequence of non-negative, identically distributed random variables on a probability space $(Y,\PP)$ which is $\psi$-mixing with coefficient $\psi(n)$ fulfilling $\sum_{n\ge 1} \psi(n)/n < \infty$. Let $F$ be the distribution function of the $(Y_n)$, that is $F(y) = \PP (Y_1\le y)$, and let for some $y_0>0$
\begin{align}
 W\coloneqq \min \left\{r\in\mathbb{N}\colon \int_{y_0}^{\infty}\left(\frac{y\left(1-F(y)\right)}{\int_{0}^y\left(1-F(t)\right)\mathrm{d}t}\right)^{r+1} \frac{1}{y}\, \mathrm{d}y<\infty\right\},
\label{eq: cond AN}
\end{align} 
where we set the $\min$ as $\infty$ if such an $r$ does not exist.
Then there exists a sequence $(d(N))$ such that 
\begin{align*}
 \lim_{N\to \infty}\frac{S_N^W(y)}{d(N)}=1, \quad \text{$\PP$-a.s.}
\end{align*}
If we set $a(y):=y/ \int_0^y (1-F(t))\mathrm{d}t$, then $d(n)$ can be set as the inverse function of $a(n)$.

Furthermore, if we denote by $M^r_N(y)$ the $r$-th maximum in $\{Y_1,\ldots, Y_N\}$ then
\begin{equation} \label{eq: M/d_n}
 \lim_{N\to\infty}\frac{M^r_N(y)}{d(N)}=0, \quad \text{$\PP$-a.s.} 
\end{equation}
for all $r> W$. 
 \end{lemma}

We now state our main results. We first consider the Birkhoff sums of an observable $f\in L^1(X,\mu)$. We prove that $S_Nf$ converges almost surely to the spatial mean of the observable if we consider the average with respect to a sequence $\alpha(n)$ growing slower than $n$ and consider more iterates for the observable. The dependence on the point $x$ is given by the number $m(N,E,x)$, as defined in \eqref{def-m}, of iterates for the observable that we have to add. We remark that the sequence $\alpha(n)$ coincides with that found in the Weak Law of Large Numbers for dynamical system which are pointwise dual ergodic and have wandering rate given by a slowly varying sequence (see \cite[Chapter 3]{aa-book}).  

\begin{theorem} \label{th-elle1}
Let $T:X\to X$ be a conservative ergodic measure-preserving transformation of the measure space $(X,\BB,\mu)$ with $\mu$ a $\sigma$-finite measure with $\mu(X)=\infty$, and let $E\in \BB$ with $\mu(E)=1$. Let's assume that the first return time to $E$ satisfies:
\begin{itemize}
\item[(i)] the sequence of random variables $(\varphi_{_E}\circ T_{_E}^{n-1})$ with $n\ge 1$ defined on $(E,\mu)$ is $\psi$-mixing with coefficient $\psi(n)$ fulfilling $\sum_{n\ge 1} \psi(n)/n < \infty$;

\item[(ii)] the super-level sets of $\varphi_{_E}$ satisfy
\[
\sum_{n\ge 1} \frac{n\, (\mu(A_{>n}))^2}{(\sum_{j=0}^{n-1} \mu(A_{>j}))^2} < \infty.
\]
\end{itemize}
Then for all $f\in L^1(X,\mu)$ and for $\mu$-a.e.\ $x\in X$
\[
\lim_{N\to \infty}\, \frac{1}{\alpha(N)}\, \sum_{k=1}^{N+m(N,E,x)}\, (f\circ T^{k-1})(x) = \int_X\, f\, \rd\mu
\]
where
\begin{equation} \label{alfa}
\alpha(n):= \frac{n}{\sum_{j=0}^{n-1}\, \mu(A_{>j})}\, .
\end{equation}
\end{theorem}

In the following we also give a slightly different formulation of Theorem \ref{th-elle1} which uses $w(N,E,x)$. The main difference is that the number of iterates for the observable is fixed and the dependence on $x$ is found in the renormalising sequence. Thus, this second formulation is more similar to Hopf's Ratio Ergodic Theorem in the spirit.

\begin{theorem} \label{th-elle1-bis}
Under the same assumptions as in Theorem \ref{th-elle1}, for all $f\in L^1(X,\mu)$ and for $\mu$-a.e.\ $x\in X$ it holds
\[
\lim_{N\to \infty}\, \frac{1}{\alpha(N-w(N,E,x))}\, \sum_{k=1}^{N}\, (f\circ T^{k-1})(x) = \int_X\, f\, \rd\mu
\]
where $w(N,E,x)$ is defined in \eqref{def-w} and $\alpha(n)$ is given in \eqref{alfa}. 
\end{theorem}

\begin{remark}\label{asym-w-m}
The asymptotic behaviour of the sequence $(w(N,E,x))$ has been considered in \cite{anomalous} for some intermittent maps of the interval with infinite measure. Following the same argument and using the relations between $m(N,E,x)$ and $w(N,E,x)$ one can prove that under the conditions of the two previous theorems
\[
\lim_{N\to \infty}\, \frac{\log m(N,E,x)}{\log N} = \lim_{N\to \infty}\, \frac{\log w(N,E,x)}{\log N} = 1
\]
for $\mu$-a.e.\ $x\in X$.
\end{remark}

We have shown that in the case of summable observables the assumptions only depend on the set $E$ and on its first return time. This is in accordance with what one could expect with Hopf's Ratio Ergodic Theorem in mind. We now study the Birkhoff sums of non-negative observables $f\not\in L^1(X,\mu)$. In this case we need assumptions also on the observable itself.

Given $T:X\to X$ and $E\in \BB$ with $\mu(E)=1$ as above, for a measurable function $f:X\to \R$, a standard approach is to introduce the induced observable $f^E$ defined on the full measure set of points $x\in E$ with finite first return time by
\begin{equation} \label{induced-obs}
f^E(x) := \sum_{k=1}^{\varphi_{_E}(x)}\, (f\circ T^{k-1})(x)\, .
\end{equation}
The induced observable simply gives the contribution to the Birkhoff sums $S_Nf$ of the excursions out of $E$ for the orbit of a point $x$. For $\mu$-a.e.\ $x\in E$ let's define the time $\tau_{_{E,x}}(N)$ of the $N$-th return to $E$, which corresponds to the Birkhoff sum of $\varphi_{_E}$ for the system $(E,T_{_E})$. We have
\begin{equation} \label{returns}
\tau_{_{E,x}}(N) := \sum_{k=1}^N\, (\varphi_{_E} \circ T_{_E}^{k-1}) (x)\, .
\end{equation}
Recalling now the notion $R_{_{E,N}}(x)$ of the number of visits to $E$ up to time $N$ along the orbit of $x$ defined in \eqref{numb-visits}, the last visit to $E$ for $x$ up to time $N$ is $T_{_E}^{R_{_{E,N}}(x)-1}(x)$ and the time of this last visit is then given by $\tau_{_{E,x}}(R_{_{E,N}}(x)-1)\le N$. Then the following relation follows for $\mu$-a.e.\ $x\in E$
\begin{equation}\label{f-vs-fe}
S_Nf(x) = \sum_{n=1}^{R_{_{E,N}}(x) -1}\, (f^E \circ T^{n-1}_{_E})(x) + \sum_{k=\tau_{_{E,x}}(R_{_{E,N}}(x)-1)+1}^{N}\, (f\circ T^{k-1})(x)
\end{equation}
with the standard convention that a summation vanishes if the upper index is smaller than the lower index. Moreover, using that $f\ge 0$ and $\tau_{_{E,x}}(R_{_{E,N}}(x))>N$, being $\tau_{_{E,x}}(R_{_{E,N}}(x))$ the time of the first visit to $E$ after time $N$, we immediately obtain
\[
\sum_{n=1}^{R_{_{E,N}}(x) -1}\, (f^E \circ T^{n-1}_{_E})(x) \le S_N f(x) \le \sum_{n=1}^{R_{_{E,N}}(x)}\, (f^E \circ T^{n-1}_{_E})(x)
\]
for $\mu$-a.e.\ $x\in E$. Thus the pointwise convergence of the Birkhoff sums of $f$ could in principle be obtained by the behaviour of the Birkhoff sums of $f^E$ for the induced system $(E,T_{_E},\mu)$. However $f\not\in L^1(X,\mu)$ implies $f^E\not\in L^1(E,\mu)$ (see for examples \cite[Lemma 2]{zwei-notes}), hence we cannot use Birkhoff's Theorem to obtain the pointwise convergence for the Birkhoff sums of $f^E$. A solution is to apply Lemma \ref{lemma-aar-nakada} to the sequence of random variables $(f^E\circ T_{_E}^{n-1})$.

The simplest case to consider is when $f^E$ is constant on the level sets $A_n$, which corresponds to $f$ constant on the level sets of the \emph{hitting time function}
\begin{equation}\label{hitting-time}
h_{_E} : X \to \N_0\, ,\quad h_{_E}(x) := \min \{ k\ge 0\, :\, T^k(x) \in E\}
\end{equation}
which is defined and finite $\mu$-a.e.\ on $X$. Setting
\begin{equation}\label{level-sets-hit}
E_n := \{ x\in X\, :\, h_{_E}(x) = n\}
\end{equation}
with $E_0=E$, for the level sets of $h_{_E}$, if $f(x) = f_k$ for all $x\in E_k$ then
\begin{equation} \label{fe-const}
f^E(x) = \sum_{k=0}^{n-1}\, f_k\, ,\quad \forall\, x\in A_n\, .
\end{equation}
We recall that $\mu(A_{>n}) = \mu(E_n)$ for all $n\ge 0$ (see \cite[Lemma 1]{zwei-notes}).

In the proof of the main result of this section, we use the following property for the waiting times between two consecutive returns to $E$.

\begin{lemma}\label{lem-bound-wt}
Let us consider the function $q:\R_{>0} \to \N$ given by 
\[
q(t) := \min \left\{ j\in \N\, :\, \mu(A_{>j}) \le \frac 1t \right\}
\]
for the super-level sets $(A_{>n})$ defined in \eqref{super-level-sets}. Then 
\[
\mu \left( \left\{ x\in E\, :\, \max_{1\le k \le R_{_{E,n}}(x)} \, (\varphi_{_E}\circ T_{_E}^{k-1})(x) \ge n\, \xi(n) \quad \text{for infinite $n$}\right\} \right) =0
\]
where
\[
\xi(n) := \frac 1n\, q \Big( \alpha(n)\, \log (\alpha(n)) \log^2 (\log (\alpha(n)))\Big)
\]
and $\alpha(n)$ is given in \eqref{alfa}.
\end{lemma}

\begin{theorem}\label{th-notelle1}
Let $T:X\to X$ be a conservative ergodic measure-preserving transformation of the measure space $(X,\BB,\mu)$ with $\mu$ a $\sigma$-finite measure with $\mu(X)=\infty$, and let $E\in \BB$ with $\mu(E)=1$. 
Let $f:X\to \R_{\ge 0}$ be a non-negative measurable function such that $f\not\in L^1(X,\mu)$ and let $f^E$ be its induced version on $E$ defined in \eqref{induced-obs}. We assume that:
\begin{itemize}
\item[(i)] the sequence of random variables $(\varphi_{_E}\circ T_{_E}^{n-1})$ with $n\ge 1$ defined on $(E,\mu)$ is $\psi$-mixing with coefficient $\psi(n)$ fulfilling $\sum_{n\ge 1} \psi(n)/n < \infty$;

\item[(ii)] there exist $g_1,g_2 : X \to \R_{\ge 0}$ constant on the level sets $(E_n)$ of the hitting time function to $E$ such that:
\begin{itemize}
\item[(ii)-(a)] $g_1(x)\le f(x)\le g_2(x)$ for $\mu$-a.e.\ $x\in X$;

\item[(ii)-(b)] the induced functions satisfy $g_i^E(n) := g_i^E|_{A_n} = n\, L_i(n)$ for $i=1,2$, where $L_i (n)$ are normalised slowly varying function (see \eqref{eq:karamata-norm}), and $L_1(n) \sim L_2(n)$ as $n\to \infty$;

\item[(ii)-(c)] given the function $\xi(n)$ defined in Lemma \ref{lem-bound-wt}, for all functions $\tilde \xi(n)$ with $c\le \tilde \xi(n) \le \xi(n)$ for some constant $c>0$, the functions $L_i$ for $i=1,2$ satisfy $L_i(n\, \tilde\xi(n))\sim L_i(n)$ as $n\to \infty$;

\item[(ii)-(d)] if $\Gamma_i(n):= \min\{ k\in \N\, :\, g_i^E(k)> n\}$ for $i=1,2$, then
\begin{equation} \label{cond-fe}
\int_{0}^{\infty} \frac{y\, \left( \mu(A_{\ge \Gamma_i(y)}) \right)^2}{\left( \int_0^y\, \mu(A_{\ge \Gamma_i(t)}) \mathrm{d}t\right)^2} \mathrm{d}y<\infty\, ;
\end{equation}

\item[(ii)-(e)] if $\ell_i(n) := \sum_{k=0}^n\, \mu(A_{\ge \Gamma_i(k)})$, then $\ell_i(n)$ for $i=1,2$ is super-slowly varying at infinity with rate function itself (see Definition \ref{def-ssv}) and
\begin{equation}\label{cond-uffa}
\ell_i\Big( \alpha(N)\Big) \sim L_i(N)\, \sum_{k=0}^N\, \mu(A_{>k})\, ,
\end{equation}
as $N\to \infty$ for $i=1,2$, where $\alpha(n)$ is defined in \eqref{alfa}.
\end{itemize}
\end{itemize}

Then there exists a function $G:\N \to \R_{\ge 0}$ which is asymptotically equivalent to $f^E$, that is there exists a set $\tilde E \subset E$ with $\mu(\tilde E)=1$ such that for every sequence $(x_n)$ with $x_n\in A_n\cap \tilde E$ for all $n\in \N$ we have $f^E(x_n) \sim G(n)$ as $n\to \infty$, and for $\mu$-a.e.\ $x\in X$
\[
\lim_{N\to \infty}\, \frac{1}{G(N)}\, \sum_{k=1}^{N}\, (f\circ T^{k-1})(x) = 1\, .
\]
\end{theorem}

In Section \ref{subsec-notelle1} we discuss the assumptions of Theorem \ref{th-notelle1} and consider particular cases for the observable $f$ and specific dynamical systems.

\section{Examples and applications} \label{sec-examples}

We first describe two classes of dynamical systems to which our results apply. Then we discuss the assumptions and the result of Theorem \ref{th-notelle1}. Some proofs of statements in this section are collected in Section \ref{app-minor} to ease the reading of the paper.

\subsection{Maps of the interval with an indifferent fixed point} \label{subsec:interval-maps}
This class of maps is considered the easiest example of non-uniformly hyperbolic systems, and thus the easiest model of a system with ``anomalous'' dynamical phenomena. Here we use the general approach introduced in \cite{zwei}. 

Let $T:[0,1]\to [0,1]$ be a piecewise monotonic $C^2$ map with respect to a partition $\PPP$ of open subintervals, that is $\PPP$ is a finite or countable family of open subintervals $(P_j)_{j\in J}$ such that $m([0,1]\setminus \cup_j P_j)=0$, being $m$ the Lebesgue measure, and $T|_{P_j}$ is $C^2$ and strictly monotonic for all $j\in J$. 

\begin{definition}\label{def-afn-maps}
A map $T:[0,1]\to [0,1]$ which is piecewise monotonic and $C^2$ with respect to a partition $\PPP=(P_j)_{j\in J}$ of open subintervals defines a \emph{basic AFN-system} if it is conservative and ergodic with respect to $m$ and if
\begin{itemize}
\item[(A)] $T''/(T')^2$ is bounded on $\cup_j P_j$;
\item[(F)] $T\PPP = \{ T(P_j) \, :\, j\in J\}$ is a finite collection;
\item[(N)] there exists a non-empty finite collection $\tilde \PPP \subset \PPP$ such that each $P_j \in \tilde \PPP$ has an \emph{indifferent fixed point} $x_j$, that is
\[
\lim_{x\to x_j,\, x\in P_j}\, T(x) = x_j \quad \text{and} \quad T'(x_j) := \lim_{x\to x_j,\, x\in P_j}\, T'(x) = 1\, ,
\]
which is a \emph{one-sided regular source}, that is $T'$ decreases on $P_j \cap (0, x_j)$ and increases on $P_j \cap (x_j,1)$. Moreover $T$ is uniformly expanding away from the indifferent fixed points, that is for all $\epsilon>0$ there exists $\rho_\epsilon>1$ such that
\[
|T'(x)|\ge \rho_\epsilon\, , \quad \forall\, x\in [0,1]\setminus \bigcup_{P_j\in \tilde \PPP}\, P_j \cap (x_j-\epsilon, x_j+\epsilon)\, .
\] 
\end{itemize}
\end{definition}

By \cite[Thm A]{zwei} a basic AFN-system admits an absolutely continuous invariant measure $\mu$ with a lower semicontinuous density $h(x)$ of the form
\[
h(x) = c(x) \,  g(x)\, , \quad \text{with}\quad g(x) = \left\{ \begin{array}{ll} \frac{x-x_j}{x-(T|_{P_j})^{-1}(x)}\, , & \text{for }\, x\in P_j\in \tilde \PPP \\[0.2cm] 1\, , & \text{for }\, x\in P_j\not\in \tilde \PPP \end{array} \right.
\]
and $c(x)$ satisfying $0< C^{-1} \le c(x) \le C$ for some constant $C>0$.

Examples of basic AFN-systems are the Pomeau-Manneville maps, the Liverani-Saussol-Vaienti maps and the Farey map. They all share the following structure. Let $\TTT$ be the family of maps of $[0,1]$ for which there is a partition $\PPP=\{ P_0,P_1\}$ of $[0,1]$ into two open intervals with respect to which they are piecewise monotonic and $C^2$, and $T(P_0)=T(P_1)=(0,1)$. Moreover the collection $\tilde \PPP$ contains only the set $P_0=(0,\bar x)$, and the maps have an indifferent fixed point at $x_0=0$ at which they satisfy
\begin{equation} \label{class-t}
T(x) = x+ C\, x^{1+p} +o(x^{1+p})\, , \quad \text{for $x\in P_0$ and $x\to 0^+$}
\end{equation}
with $p\ge 1$ for some constant $C>0$, and are uniformly expanding away from $x_0$. 

By \cite[Lemma 8]{zwei} there is a set $E\subset [0,1]$ with $\mu(E)<\infty$ such that the induced map $T_{_E}$ is piecewise monotonic and $C^2$ with respect to a partition $\PPP_{E}$, satisfies conditions (A) and (F), and is uniformly expanding (U). It is then called an AFU-system. Finally \cite[Thm 1]{aar-nak-mix} gives that if $T_{_E}$ is weakly mixing and $\inf_{n\ge 1, Q\in \PPP^n_{E}} m(T_{_E}^n (Q)) >0$, then the sequence of random variables $(\varphi_{_E}\circ T_{_E}^{n-1})$ with $n\ge 1$ defined on $(E,\mu)$ is $\psi$-mixing with coefficient $\psi(n)$ fulfilling $\psi(n) \le K\, \theta^n$ for some constants $K>0$ and $\theta\in (0,1)$. Thus assumption (i) of Theorems \ref{th-elle1} and \ref{th-notelle1} is satisfied.

For the maps in the family $\TTT$ one can choose $E=P_1$ to obtain an exact AFU-system $T_{_E}$ which satisfies the assumptions of \cite[Thm 1]{aar-nak-mix}.

Moreover we need information on the super-level sets $(A_{>n})$.

\begin{definition}\label{nice-expansions}
We say that a basic AFN-system $T$ admits \emph{regularly varying nice expansions} if for all $P_j\in\tilde \PPP$ there are $p_j\geq 1$ and $L_j$ slowly varying functions at infinity (see Appendix \ref{sec:svf}) such that 
\[
 T(x)\sim x+ \left|x-x_j\right|^{1+p_j}\, \left(L_j\left(1/\left|x-x_j\right|\right)\right)^{-1}\, , \quad \text{for $x\in P_j$ and $x\to x_j$.}
\]
\end{definition}

Using then computations similar to those in \cite[Lemma 2]{thaler}, one can show that a basic AFN-system with regularly varying nice expansions satisfies assumption (ii) of Theorem \ref{th-elle1} if $p:= \max_{P_j \in \tilde \PPP} p_j =1$ and $L(y):= \max_{P_j\in \tilde \PPP} |L_j(y)|$ satisfies
\begin{equation} \label{ass-L}
\int_{1}^\infty\, \left(\frac{L(y)}{\int_1^y\frac{L(t)}{t} \mathrm{d}t}\right)^2\, \frac 1y \,\mathrm{d}y < \infty\, .
\end{equation}
Indeed if $p=1$ the super-level sets satisfy $\mu(A_{>n}\cap P_j) \sim |L_j(n)|/n$ as $n\to \infty$ for all $P_j\in \tilde \PPP$, so that by \eqref{ass-L} assumption (ii) is satisfied.

Maps in $\TTT$ have $L_0(y)\equiv C$, so that the previous condition is easily satisfied.

\subsection{Fibred systems related to multi-dimensional continued fraction expansions} \label{subsec-fibred}
The class of systems in dimension greater than one for which the $\psi$-mixing property has been studied includes the so-called \emph{fibred systems} on compact and connected subsets $E$ of $\R^d$.

The notion of fibred system was introduced in \cite{schw} in relation to the theory of real number expansions. Here we use results from \cite{nak-nat}. We first recall the definition of a fibred system.

\begin{definition} \label{def-fibred}
    Let $E$ be a compact and connected subset of $\R^d$, with the
    Borel $\sigma$-algebra $\BB$ and let $m$ denote the
    $d$-dimensional normalised Lebesgue measure on $E$. Let $V$ be a
    measurable map of $E$ onto itself. The pair $(E,V)$ is called a
    \emph{fibred system} if it satisfies the following properties.
    \begin{itemize}
      \item[(h1)] There exists a finite or countable measurable
        partition $\CC=\{C_j\}_{j\in J}$ of $E$ into open sets such that the
        restriction of $V$ to $C_j$ is injective for all $j\in J$.
      \item[(h2)] The map $V$ is measurable and non-singular. To simplify the next conditions we assume that $V|_{C_j}$ is differentiable for all $j\in J$.
    \end{itemize}
    For $j\in J$, we denote by $\phi_j$ the inverse of the
    restriction $V|_{C_j}$. Given the iterated
    partition $\CC^n = \bigvee_{k=0}^{n-1}\, V^{-k} \CC$, we denote by $\phi_{j_1,\,\dots,\,j_n}$ the local inverse of
    $V^n$ restricted to $C_{j_1,\,\dots,\,j_n} \in \CC^n$.    
    \begin{itemize}
      \item[(h3)] There exists a sequence $(\sigma(n))_{n\geq 0}$ with
        $\sigma(n)\rightarrow 0$ as $n\rightarrow \infty$ and such
        that
        \[
            \sup_{(j_1,\,\dots,\,j_n)\in J^n}\,\diam
            C_{j_1,\,\dots,\,j_n}\le \sigma(n).
        \]
      \item[(h4)] There exist a finite number of measurable subsets
        $U_1,\,\ldots,\,U_N$ of $E$ such that for any cylinder
        $C_{j_1,\,\dots,\,j_n}$ of positive measure, there exists
        $U_i$ with $1\le i\le N$ such that $V^n(C_{j_1,\,\dots,\,j_n})
        = U_i$ up to measure-zero sets.
      \item[(h5)] There exists a constant $\lambda\ge 1$ such that
        \[
            \esssup_{V^n(C_{j_1,\,\dots,\,j_n})} |D\phi_{j_1,\,\dots,\,j_n}|
            \le \lambda\,\essinf_{V^n(C_{j_1,\,\dots,\,j_n})}
            |D\phi_{j_1,\,\dots,\,j_n}|
        \]
        where $D\phi_{j_1,\dots,j_n}$ denotes the Jacobian determinant
        of $\phi_{j_1,\,\dots,\,j_n}$.
      \item[(h6)] For any $1\le i\le N$, $U_i$ contains a proper
        cylinder.
      \item[(h7)] There is a constant $r_1 >0$ such that
        \[
            \left|D\phi_{j_1,\,\dots,\,j_n} (p_1) - D\phi_{j_1,\,\dots,\,j_n}
              (p_2) \right| \le r_1 m(C_{j_1,\,\dots,\,j_n}) \| p_1 - p_2\|
        \]
        for any $p_1,p_2 \in U_i$ and all $i$.
      \item[(h8)] There is a constant $r_2 >0$ such that
        \[
            \left\| \phi_{j_1,\,\dots,\,j_n} (p_1) - \phi_{j_1,\,\dots,\,j_n}
              (p_2) \right\| \le r_2 \sigma(n)\| p_1 - p_2\|
        \]
        for any $p_1,p_2 \in U_i$ and all $i$.
      \item[(h9)] Let $\mathcal F$ be a finite partition generated by
        $U_1,\,\dots,\,U_N$ and denote by ${\mathcal F}^c _\ell$ the
        cylinders in $\CC^\ell$ that are not contained in any element of
        $\mathcal F$. Then, as $\ell\rightarrow \infty$
        \[
            \gamma(\ell) \coloneqq \sum_{C(j_1,\,\dots,\,j_\ell) \in {\mathcal
                F}^c _\ell}\, m(C(j_1,\,\dots,\,j_\ell)) \rightarrow 0.
        \]
    \end{itemize}
\end{definition}

By \cite[Thm 2]{nak-nat}, a fibred system $(E,V)$ is $\psi$-mixing, with respect to the iterated partition $\CC^n$, and there exist constants $K>0$ and $\theta\in (0,1)$ such that
\begin{equation} \label{stima-psi}
\psi(n) \le K\, \left( \theta^{\sqrt{n}} + \sigma(\sqrt{n}) + \gamma(\sqrt{n})\right)\, ,
\end{equation}
where $\sigma(k)$ and $\gamma(k)$ are defined in (h3) and (h9).

Examples of fibred systems are obtained in the study of multi-dimensional continued fractions algorithms. Here we consider the two-dimensional system studied in \cite{mio}, which is related to the two-dimensional continued fractions algorithm introduced in \cite{garr}. Let
\[
    \triangle \coloneqq \set{(x,y)\in \R^2 \,:\, 1 \ge x \ge y \ge 0}
\]
and define the partition $\{\Gamma_0, \Gamma_1\}$ of $\triangle$ by
\[
    \Gamma_0 \coloneqq \set{(x,y)\in \R^2 \,:\, 1 \ge
    x \ge y > 1-x},
\]
and
\[
    \Gamma_1 \coloneqq  \set{(x,y)\in \R^2 \,:\,
    1-y \ge x\ge y\ge 0}.
\]
We consider the map $S:\triangle \rightarrow \triangle$ defined as
\[
    S(x,y) \coloneqq
    \begin{cases}
        \left(\frac yx,\,\frac{1-x}{x}\right) & \text{if }(x,y)\in\Gamma_0\\[0.2cm]
        \left(\frac{x}{1-y},\,\frac{y}{1-y}\right) & \text{if
        }\,(x,y)\in\Gamma_1
    \end{cases}.
\]
In \cite{mio} it is proved that $S$ admits an absolutely continuous invariant measure $\mu$ with density $h(x,y)=\frac{1}{xy}$, so that $\mu(\triangle)=\infty$, and that the system $(\triangle, S,\mu)$ is ergodic and conservative. 

The next result shows that there exists a set $E\subset \triangle$ for which assumption (i) of Theorem \ref{th-elle1} holds. The proof of the proposition is in Section \ref{app-minor}.

\begin{proposition}\label{prop-2dfarey}
Let $E\subset \triangle$ be the triangle with vertices
$Q_1=\left(\frac 12, \frac 12\right)$, $Q_2= \left(\frac 23, \frac 13\right)$ and $Q_3=(1,1)$, with the sides
$Q_1Q_2$ and $Q_2Q_3$ not included. Then $\mu(E)<\infty$, the induced map $S_{_E}$ on $E$ is $\psi$-mixing with respect to the level sets $(A_n)$ of the return time function $\varphi_{_E}$, and $\sum_{n\ge 1}\, \psi(n)/n <\infty$.
\end{proposition}

Finally, assumption (ii) of Theorem \ref{th-elle1} is indirectly studied in \cite[Appendix B]{mio} but not proved to hold. Its validity implies the wandering rate $w_n(E):=\sum_{k=0}^{n}\, \mu(A_{>n})$ of $E$ to be slowly varying as discussed in \cite{aar-nakada}. This property of $w_n(E)$ is studied in \cite{mio} where it is proved that $w_n(E) \asymp (\log n)^2$, from which one could conjecture $\mu(A_{>n}) = O(\frac{\log n}n)$, and assumption (ii) would follow. Unfortunately we don't have a proof of this asymptotic behaviour for $\mu(A_{>n})$.

\subsection{Discussions on Theorem \ref{th-notelle1}} \label{subsec-notelle1}

As previously remarked, in the case of non-summable observables the pointwise asymptotic behaviour of the Birkhoff sums depends on some properties of the observable itself. We have been able to show that for $f\ge 0$ it holds $S_Nf(x) \sim G(N)$ for $\mu$-a.e.\ $x\in X$, where $G(n)$ is the asymptotic behaviour of the induced observable $f^E$ on the level sets $(A_n)$, if $f$ can be bounded by functions $g_1$ and $g_2$ with induced functions $g_i^E$ of the form $nL_i(n)$ for suitable normalised slowly varying functions (see Appendix \ref{sec:svf} for definitions and details on slowly varying functions). The assumptions on the functions $L_i(n)$ concern their asymptotic behaviour, as in assumptions (ii)-(b) and (c), and their link with the first return time function to the set $E$, as in assumptions (ii)-(d) and (e).

First we discuss assumptions (ii)-(b) and (c), showing sufficient conditions for $f$ and for a slowly varying function to satisfy them. Then we discuss assumptions (ii)-(d) and (e) in a specific example, and consider the necessity of (ii)-(e). Further we consider the existing literature on this problem.

\begin{proposition} \label{prop-normslow}
Let $f:X\to \R_{\ge 0}$ be constant on the level sets $(E_n)$ of the hitting time function \eqref{hitting-time} and assume that the induced function $f^E$ satisfies $f^E|_{A_n} = n\, L(n)$ for all $n\ge 0$, with $L$ slowly varying. Then there exist $g_1,g_2:X\to \R_{\ge 0}$ satisfying (ii)-(a) and (b) of Theorem \ref{th-notelle1}.
\end{proposition}

\begin{proof}
Apply Lemma \ref{normslow} to $f^E$ to find $g_1^E$ and $g_2^E$. Then $g_1$ and $g_2$ are obtained using relation \eqref{fe-const} between a function which is constant on the level sets $(E_n)$ and its induced function.
\end{proof}

We can thus reduce our attention to the case of a function $f$ with induced function which is asymptotically equivalent to a function $G(n)=n\, L(n)$ with $L$ normalised slowly varying. Hence using normalised slowly varying functions does not reduce generality.

Let us now consider (ii)-(c). First notice that if $\xi(n)$ is bounded then assumption (ii)-(c) is easily obtained by the properties of slowly varying functions. Moreover, if $h(n)$ is a function for which $\xi(n) \le h(n)$ and assumption (ii)-(c) is satisfied for $h$, then it is immediately satisfied also for $\xi$.

Fixed a slowly varying function $h$, in Lemma \ref{remarkii} we give a sufficient condition on slowly varying functions $L$ to satisfy $L(n\, \tilde h(n)) \sim L(n)$ for all $\tilde h$ such that $c\le \tilde h(n) \le h(n)$. Notice that in Lemma \ref{remarkii} we assume that $h$ is bounded below from 1. If $h$ does not satisfy this assumption, simply apply the lemma to $k(n):=\min\{ 1+h_0, h(n)\}$ for some $h_0>0$, which is still slowly varying.

To study assumption (ii)-(c) using Lemma \ref{remarkii} for the function $\xi$ it is thus sufficient to show that it is slowly varying or that there exists a slowly varying function $h$ such that $\xi(n)\le h(n)$. It follows from computations in Proposition \ref{notelle1-class-t} that $\xi(n) \sim \log^2 (\log n)$, and therefore is slowly varying, for the class $\TTT$ of maps of the interval considered in Section \ref{subsec:interval-maps}. For the map studied in Section \ref{subsec-fibred} instead, it is shown in \cite{mio} that $\mu(A_{>n}) = O(\frac{\log^2 n}{n})$ and $\alpha(n) = O(n/\log^2 n)$, hence in Lemma \ref{lem-bound-wt} one finds $q(t) =O(t\, \log^2 t)$, and
\[
\xi(n) = O\Big( \log n \, \log^2 (\log\, n)\Big)\, .
\]
Thus we estimate $\xi$ with a slowly varying function and Lemma \ref{remarkii} gives a sufficient condition on the slowly varying function $L$ to satisfy assumption (ii)-(c). Hence if the measure of the super-level sets $(A_{>n})$ is bounded by a regular varying function with index -1, and the sequence $\alpha(n)$ defined in \eqref{alfa} is bounded by a regular varying function with index 1, then $\xi$ is bounded by a slowly varying function. In Remark \ref{wand-rate-slowly-var} we show that it is the case if \eqref{cond-fe} holds.

Finally we discuss the relation between the observable $f$ and the first return time function to $E$. In particular we show that assumptions (ii)-(d) and (e) simplify when the two limiting functions $g_1$ and $g_2$ have induced functions which increase asymptotically linearly. The proof of the following proposition is in Section \ref{app-minor}.

\begin{proposition}\label{notelle1-linear}
Let $T$, $(X,\BB,\mu)$ and $E$ be defined as in Theorem \ref{th-notelle1}. Let the sequence of random variables $(\varphi_{_E}\circ T_{_E}^{n-1})$ with $n\ge 1$ satisfy assumption (i) of Theorem \ref{th-notelle1}. Let $f:X\to \R_{\ge 0}$ be a non-negative measurable function such that $f\not\in L^1(X,\mu)$ and let $f^E$ be its induced version on $E$ defined in \eqref{induced-obs}. We assume that:
\begin{itemize}
\item[(c1)] there exist $g_1,g_2 : X \to \R_{\ge 0}$ constant on the level sets $(E_n)$ of the hitting time function to $E$ such that $g_1(x)\le f(x)\le g_2(x)$ for $\mu$-a.e.\ $x\in X$, and the induced functions satisfy $g_i^E(n) \sim n$ for $i=1,2$;

\item[(c2)] the super-level sets of $\varphi_{_E}$ satisfy
\[
\sum_{n\ge 1} \frac{n\, (\mu(A_{>n}))^2}{(\sum_{j=0}^{n-1} \mu(A_{>j}))^2} < \infty;
\]

\item[(c3)] if $\ell(n) := \sum_{k=0}^n\, \mu(A_{\ge k})$, then $\ell( N/\ell(N)) \sim \ell(N)$ as $N\to \infty$.
\end{itemize}
Then for $\mu$-a.e. $x\in X$ it holds
\[
\lim_{N\to \infty}\, \frac{1}{N}\, \sum_{k=1}^{N}\, (f\circ T^{k-1})(x) = 1\, .
\]
\end{proposition}

\begin{corollary} \label{notelle1-class-t}
Let us consider a map of the interval $T\in \TTT$ with only one indifferent fixed point at $0$ satisfying \eqref{class-t} with $p=1$ and $C=1$, and let $\mu$ be the infinite absolutely continuous $T$-invariant measure. If $f:X\to \R_{\ge 0}$ is constant on the level sets $(E_n)$ of the hitting time function \eqref{hitting-time} and the induced function $f^E$ satisfies $f^E|_{A_n} \sim n$ as $n\to \infty$ then
\begin{equation} \label{birk-av}
\lim_{N\to \infty}\, \frac 1N\, \sum_{n=1}^N\, (f\circ T^{n-1})(x) =1\, , \quad \text{for $\mu$-a.e.\ $x\in [0,1]$.}
\end{equation} 
\end{corollary}

\begin{proof}
This is a special case of Proposition \ref{notelle1-linear}. Assumption (i) of Theorem \ref{th-notelle1} holds as discussed in Section \ref{subsec:interval-maps}, and (c1) follows by definition of the observable $f$. Finally for $T\in \TTT$ with $p=1$ and $C=1$ it holds $\mu(A_{>n}) \sim \frac 1n$. Thus $\ell(n) \sim \log n$ and (c2)-(c3) hold true.
\end{proof}

In order to not only considering the trivial case of $G(n)\sim n$, we give a more involved example generalising Corollary \ref{notelle1-class-t}.
\begin{example}
 Let $(X,\mathcal{B},\mu)$ and $E\in\mathcal{B}$ be a system fulfilling (i) of Theorem \ref{th-notelle1} with $\mu(A_{>n})\sim \log^b(\log n)/n$ as $n\to \infty$ for $b\in\R$. 
 Furthermore we assume $f:X\to \R_{\ge 0}$ to be constant on the level sets $(E_n)$ of the hitting time function \eqref{hitting-time} and let the induced function $f^E$ satisfy $f^E|_{A_n} \sim n\, \log^c(\log n)$ as $n\to \infty$ with $c\in\mathbb{R}$, then
\begin{equation} \label{birk-av1}
\lim_{N\to \infty}\, \frac {1}{N\, \log^c(\log N)}\, \sum_{n=1}^N\, (f\circ T^{n-1})(x) =1\, , \quad \text{for $\mu$-a.e.\ $x\in X$.}
\end{equation}
We will use Theorem \ref{th-notelle1} to verify this statement. Obviously, conditions (i), (ii)-(a), and (ii)-(b) are already fulfilled by assumption. 
To verify condition (ii)-(c) we notice that $q(n)\sim n\, \log^b(\log n)$ (using \cite[Cor.\ 2.3.4]{reg-var-book}). Furthermore $\alpha(n)\sim n/( \log n\, \log^{b}(\log n))$ implies
$\xi(n)\sim n\, \log^2(\log n)$ and $\log^c(\log (n\, \xi (n)))\sim \log^c(\log n)$ gives condition (ii)-(c). 

To verify condition (ii)-(d) we notice that $\Gamma(n)\sim n/ \log^c(\log n)$ (also by applying \cite[Cor.\ 2.3.4]{reg-var-book}) and so 
$\mu(A_{> \Gamma(n)})\sim \log^{b+c}(\log n)/n$.
Thus 
\begin{align*}
 \int_{0}^{\infty} \frac{y\, \left( \mu(A_{\ge \Gamma(y)}) \right)^2}{\left( \int_0^y\, \mu(A_{\ge \Gamma(t)}) \mathrm{d}t\right)^2} \mathrm{d}y
 &\sim \int_{0}^{\infty} \frac{\log^{2(b+c)}(\log y)}{y \left( \int_0^y\, \log^{b+c}(\log t)/t) \mathrm{d}t\right)^2} \mathrm{d}y
 \sim \int_{0}^{\infty} \frac{1}{y\, \log^2 y} \mathrm{d}y<\infty.
\end{align*}
Finally, we are left to verify condition (ii)-(e). We have $\ell(n)\sim \log n\, \log^{b+c}(\log n)$ which is super-slowly varying with rate function itself and $\ell(\alpha(n))\sim \ell(n)$.
On the other hand 
\[
L(n)\, \sum_{k=1}^n\mu(A_{>k})\sim \Big(\log^c(\log n)\Big)\, \Big(\log n\, \log^b(\log n)\Big)\sim \ell(n)
\]
giving all properties of Theorem \ref{th-notelle1} and \eqref{birk-av1} holds. 
\end{example}

In the last examples we have seen that not only the trivial case $G(n)\sim n$ and $\mu(A_{>n})\sim \frac{1}{n}$ is applicable on Theorem \ref{th-notelle1}. However, as we will see in the next proposition, condition (ii)-(e) can be quite restrictive and, as we will see, we will not be able to drop it. 

\begin{proposition}\label{notelle1-counterex}
 Let $T$, $(X,\BB,\mu)$ and $E$ be defined as in Theorem \ref{th-notelle1}. Let the sequence of random variables $(\varphi_{_E}\circ T_{_E}^{n-1})$ with $n\ge 1$ satisfy assumption (i) of Theorem \ref{th-notelle1}. Let $f:X\to \R_{\ge 0}$ be a non-negative measurable function such that $f\not\in L^1(X,\mu)$ and let $f^E$ be its induced version on $E$ defined in \eqref{induced-obs}. We assume that:
\begin{itemize}
\item[(w1)] there exist $g_1,g_2 : X \to \R_{\ge 0}$ constant on the level sets $(E_n)$ of the hitting time function to $E$ such that $g_1(x)\le f(x)\le g_2(x)$ for $\mu$-a.e.\ $x\in X$, and the induced functions satisfy $g_i^E(n) \sim n\,\log n$ for $i=1,2$;

\item[(w2)] the super-level sets of $\varphi_{_E}$ satisfy $\mu(A_{>n})\sim \frac{1}{n}$.
\end{itemize}
Then there exist $u_1<u_2$ such that for $\mu$-a.e. $x\in X$ it holds
\[
\limsup_{N\to \infty}\, \frac{1}{N\,\log N}\, \sum_{k=1}^{N}\, (f\circ T^{k-1})(x) >  u_2\quad\text{ and }\quad
\liminf_{N\to \infty}\, \frac{1}{N\,\log N}\, \sum_{k=1}^{N}\, (f\circ T^{k-1})(x) <  u_1\, .
\]
\end{proposition}
We will first assure ourselves that Proposition \ref{notelle1-counterex} fulfils all assumptions of Theorem \ref{th-notelle1} with $G(n)\sim n\,\log n$ except for (ii)-(e). It is clear that assumptions (ii)-(a) and (b) are satisfied by (w1). To verify (ii)-(c), we note that (w2) implies $\alpha(n)\sim n/\log n$ from which we obtain $\xi(n)\sim \log^2 (\log n)$
and $\log(n\,\xi(n))\sim \log n$ verifying (ii)-(c). To verify (ii)-(d) we first notice that $\Gamma(n)\sim n/\log n$ is an asymptotic inverse of $G(n)$ which we either get by direct calculation or by using e.g.\ \cite[Cor.~2.3.4]{bgl}. This implies that $\mu(A_{>\Gamma(n)})\sim \log n/n$ and by an easy calculation we obtain that \eqref{cond-fe} is fulfilled. The proof of Proposition \ref{notelle1-counterex} is in Section \ref{app-minor}.

\vskip0.2cm
Finally, we discuss the existing literature. The pointwise convergence of Birkhoff averages for non-summable observables and infinite mea\-sure-preserving dynamical systems has been recently discussed in \cite{lenci}. The authors of \cite{lenci} have proved that \eqref{birk-av} holds under strong assumptions on the observable $f$ and very weak assumptions on the system. Basically they considered $L^\infty$ observables for which it is possible to control in a very strong way the contributions to the Birkhoff sums from the excursions out of a finite measure set $E$, the set on which we have considered the induced map. In this paper we have less stringent assumptions on the observables but have considered more particular dynamical systems, those for which an induced map exists which is $\psi$-mixing. We finish this section by discussing the relations of our results with those in \cite{lenci} for the Farey map, a system in the class $\TTT$. 

\begin{remark} \label{global-obs}
Let us consider the particular case of the Farey map $F:[0,1]\to [0,1]$ which belongs to the class $\TTT$ considered in Section \ref{subsec:interval-maps} and is defined as
\[
F(x) = \left\{ \begin{array}{ll} \frac{x}{1-x}\, , & \text{if $x\in \left[0,\frac 12\right]$,}\\[0.2cm] \frac{1-x}{x}\, , & \text{if $x\in \left[\frac 12,1\right]$.}
\end{array} \right.
\]
It is well known that the absolutely continuous invariant measure is $d\mu(x) = \frac 1{x \log 2} \, dx$. If $E=(\frac 12,1)$, for which $\mu(E)=1$, it holds $E_n= (\frac{1}{n+1},\frac 1n)$. Thus $\mu(E_n) = \mu(A_{>n}) = \log_2(1+\frac 1n)\sim \frac{1}{n\log 2}$. 

A particular class of observables studied in the context of infinite mixing theory is that of \emph{global observables}. For maps in $\TTT$, we recall that a global observable is a function $f\in L^\infty([0,1],\mu)$ such that the limit
\[
\lim_{a\to 0^+}\, \frac{1}{\mu(a,1)}\, \int_a^1\, f\, d\mu
\]
exists and is finite (see \cite{bgl,bl} for more details).
We first show that observables $f\in L^\infty(X,\mu)$ which are constant on the level sets $(E_n)$ with $f_n:=f|_{E_n}$ and for which the induced function satisfies $f^E|_{A_n} \sim n$ as $n\to \infty$, are global observables. The definition of global observable for the Farey map and for observables constant on the sets $E_n$ reduces to show that the limit
\begin{equation} \label{it-is-go}
\lim_{n\to \infty}\, \frac{1}{\mu\left(\frac{1}{k+1},1\right)}\, \int_{\frac{1}{k+1}}^1\, f(x)\, d\mu = \lim_{n\to \infty}\, \frac{1}{\sum_{k=1}^n\, \log_2(1+\frac 1k)}\, \sum_{k=1}^n\, f_k\, \log_2\left(1+\frac 1k\right)
\end{equation}
exists and is finite. To prove this, it is enough to recall that $\sum_{k=1}^n\, \log_2(1+\frac 1k)\sim \log_2 n$ and to apply Abel's summation formula to get
\[
\sum_{k=1}^n\, f_k\, \log_2\left(1+\frac 1k\right) \sim \left( \sum_{k=1}^n\, f_k \right)\,  \log_2\left(1+\frac 1n\right) + \frac{1}{\log 2} \int_1^n\, \left( \sum_{k=1}^t\, f_k \right)\, \frac{1}{t^2+t}\, dt \sim
\]
\[
\sim f^E|_{A_n}\, \log_2\left(1+\frac 1n\right) + \int_1^n\,  \frac{f^E|_{A_t}}{\log 2}\, \frac{1}{t^2+t}\, dt \sim n\, \log_2\left(1+\frac 1n\right) +\frac{1}{\log 2} \int_1^n\,  \frac{t}{t^2+t}\, dt \sim \log_2(n+1)\, .
\]
On the other hand let's consider a global observable $f:X\to \R_{\ge 0}$ which is constant on the level sets $(E_n)$ with $f_n:=f|_{E_n}$. By rescaling we can assume that the limit in \eqref{it-is-go} exists and is equal to 1. In general it is not possible to conclude about the asymptotic behaviour of $f^E|_{A_n}$, however some sufficient conditions can be obtained by the following argument (see \cite{moricz1}). With $\ell_n := \sum_{k=1}^n\, \log_2(1+\frac 1k)$ let
\[
\tau_n := \frac{1}{\ell_n}\, \sum_{k=1}^n\, f_k\, \log_2\left(1+\frac 1k\right)\, .
\]
Then $f_1=\ell_1\tau_1$ and
\[
f_n = \frac{1}{\log_2\left(1+\frac 1n\right)}\, (\ell_n \tau_n - \ell_{n-1} \tau_{n-1})\, , \quad \text{for $n\ge 2$,}
\]
from which
\[
f^E|_{A_n} = \sum_{k=1}^n\, f_k = \frac{\ell_n \tau_n}{\log_2\left(1+\frac 1n\right)} - \sum_{k=1}^{n-1}\, \ell_{k}\tau_{k} \left(\frac{1}{\log_2\left(1+\frac{1}{k+1}\right)} - \frac{1}{\log_2\left(1+\frac 1k\right)}\right)\, .
\]
Since $\ell_n \sim \log_2 n$, it follows for example that if $\tau_n \sim 1 + o((\log_2 n)^{-1})$ then $f^E|_{A_n} \sim n$. Hence we can check if Proposition \ref{notelle1-class-t} is applicable to a given global observable.

Let us now compare the conditions on $f$ in Proposition \ref{notelle1-class-t} with the statement in \cite[Theorem 2.5]{lenci}. In our setting and with proper norming, the assumptions in \cite[Theorem 2.5]{lenci} can be stated as follows: $f\in L^{\infty}(X, \mu)$ and for all $\epsilon>0$ there exist $N, K\in\N$ such that for all $x\in \bigcup_{k\geq K} E_k$ we have that
\begin{align}
 \left|\frac{\sum_{j=1}^{N} \left(f\circ T^{j-1}\right)(x)}{N}-1\right|<\epsilon.\label{eq: cond LM}
\end{align}

First we show that if the assumptions for an observable $f$ in \cite[Theorem 2.5]{lenci} are satisfied, then we obtain a good asymptotic behaviour for the induced observable $f^E$, so that we are basically in a good situation to apply Theorem \ref{th-notelle1}.

\begin{lemma} \label{cfr-lenci}
 Let $f:[0,1]\to \R_{\ge 0}$ be in $L^\infty(X,\mu)$ and such that for all $\epsilon>0$ there exist $N, K\in\N$ such that for all $x\in \bigcup_{k\geq K} E_k$ we have that \eqref{eq: cond LM} holds. Then $f^E|_{A_k}\sim k$ for $k\to \infty$. 
\end{lemma}
\begin{proof}
Let $\epsilon>0$ be given and consider $N,K$ fixed as in \eqref{eq: cond LM}. We set $f_N(x)\coloneqq \sum_{j=1}^{N}(f\circ T^{j-1})(x)$. For $x\in A_R$ with $R>K$, writing $R=K+r N + s$ with $r\in \N_0$ and $0\le s \le N-1$, we get
\[
f^E(x) = \sum_{j=1}^{R+1}\, (f\circ T^{j-1})(x) = \sum_{i=1}^{r}\, f_N (T^{N(i-1)}(x)) + \sum_{j=rN+1}^{R+1}\, (f\circ T^{j-1})(x)\, .
\]
Since $T^{N(i-1)}(x) \in \bigcup_{k\geq K} E_k$ for all $i=1,\dots,r$, we apply \eqref{eq: cond LM}, and the fact that $f\in L^\infty(X,\mu)$ and $f\ge 0$, to obtain
\[
rN\, (1-\epsilon) \le f^E(x) \le rN\, (1+\epsilon) + (N+K)\, \| f\|_\infty\, , \quad \forall\, x\in A_R\, .
\]
Since $R\sim rN\to \infty$ as $R\to \infty$, the lemma is proved.
\end{proof}

Finally we give two examples to which Theorem \ref{th-notelle1} is applicable, but which fail to fulfil the conditions in \cite[Theorem 2.5]{lenci}. 
The first is an example with $f\notin L^{\infty}(X,\mu)$, the second one with $f\in L^\infty(X,\mu)$ but not satisfying \eqref{eq: cond LM}. 

\begin{example} \label{ex1}
Let us consider a strictly monotonic sequence of natural numbers $\gamma_n$ defined by $\gamma_1=4$ and $\gamma_{n+1}=\gamma_n+\left\lfloor \gamma_n^{1/2}\right\rfloor$. Then we set $f$ to be constant on the intervals $E_k$ with 
 \begin{align*}
  f(x)=\begin{cases}\gamma_n-\gamma_{n-1}\, ,& \text{if $x\in E_{\gamma_n}$ for some $n\in\N$;}\\[0.2cm]
        0\, ,&\text{else}. 
       \end{cases}
 \end{align*}
 Since $\gamma_n-\gamma_{n-1}$ tends to infinity, $f\notin L^{\infty}(X,\mu)$. On the other hand the induced function satisfies $f^E|_{A_k}= \gamma_n$ for $\gamma_n \le k< \gamma_{n+1}$, so that for all $k\in \N$ and all $x\in A_k$ we have $k+1-\sqrt{k+1}\leq  f^E(x)\leq k$, implying that $f^E|_{A_k}\sim k$ as $k\to \infty$. 
\end{example}

\begin{example} \label{ex2}
First we define a strictly monotonic sequence of even numbers $\kappa_n$ by setting $\kappa_1=4$ and $\kappa_{n+1}=\kappa_n+2\left\lfloor \kappa_n^{1/2}/2\right\rfloor$.
  Furthermore, we set 
  \begin{align*}
   U\coloneqq \left\{\kappa_n,\kappa_n+1,\ldots, \kappa_n+1/2\left(\kappa_{n+1}-\kappa_n\right),\,
    n\in\mathbb{N}\right\}.
  \end{align*}
  We define 
  \begin{align*}
  f(x)=\begin{cases}2\, , &\text{if  $x\in E_{k}$ and $k\in U$;}\\[0.2cm]
        0\, ,&\text{else}. 
       \end{cases}
 \end{align*}
 Then $0\le f(x)\leq 2$ and thus $f\in L^{\infty}(X, \mu)$. Furthermore, if $x\in A_k$ then $k-2\sqrt{k}\leq  f^E(x)\leq k+2\sqrt{k}$ implying that $f^E|_{A_k}\sim k$ as $k\to \infty$.
 
 On the other hand, for each $K\in \mathbb{N}$ there is a positive measure set $A\subset \cup_{k\ge K}E_k$ such that for all $x\in A$ we have that $f(x)=\ldots =\left(f\circ T^K\right)(x)=0$ implying that \eqref{eq: cond LM} cannot be fulfilled.
\end{example}
\end{remark}

\section{Proofs of the main results} \label{proofs}

\subsection{Proof of Theorem \ref{th-elle1}} We prove the theorem for $\mu$-a.e.\ $x\in E$. The result for $\mu$-a.e.\ $x\in X$ follows since $T$ is conservative and ergodic. Let $f\in L^1(X,\mu)$ and let $R_{_{E,N,m}}(x) := R_{_{E,N+m(N,E,x)}}(x)$ denote the number of visits to $E$ in the first $N+m(N,E,x)$ steps of the orbit of $x$. Thus for each $N$ we add a number of steps depending on $x$ and being non-decreasing in $N$.

\begin{lemma}\label{lem-hopf}
For $\mu$-a.e.\ $x\in E$ it holds
\[
\lim_{N\to \infty}\, \frac{1}{R_{_{E,N,m}}(x)}\,  \sum_{k=1}^{N+m(N,E,x)}\, (f\circ T^{k-1})(x) = \int_X\, f\, \rd\mu\, .
\]
\end{lemma}

\begin{proof}
We apply Hopf's Ratio Ergodic Theorem to $f(x)$ and $\UU_{_E}(x)$ and obtain for $\mu$-a.e.\ $x\in E$
\[
\lim_{N\to \infty}\, \frac{\sum_{k=1}^{N+m(N,E,x)}\, (f\circ T^{k-1})(x)}{R_{_{E,N,m}}(x)} =
\]
\[
= \lim_{N\to \infty}\, \frac{\sum_{k=1}^{N+m(N,E,x)}\, (f\circ T^{k-1})(x)}{(\UU_{_E}\circ T^{N+m(N,E,x)})(x)+\sum_{k=1}^{N+m(N,E,x)}\, (\UU_{_E}\circ T^{k-1})(x)} =
\]
\[
= \frac{\int_X\, f\, d\mu}{\int_X\, \UU_{_E}\, d\mu} = \int_X\, f\, \rd\mu
\]
where we have used that $\sum_{k=1}^{N+m(N,E,x)}\, (\UU_{_E}\circ T^{k-1})(x)$ is divergent.
\end{proof}

Let us now recall for $\mu$-a.e.\ $x\in E$ the definition of the time $\tau_{_{E,x}}(N)$ of the $N$-th return to $E$ given in \eqref{returns}. Then as in \eqref{trimmed-birk-sum}, we consider its trimmed version
\begin{equation} \label{trimmed-returns}
\tau^1_{_{E,x}}(N):= \tau_{_{E,x}}(N) - \max_{1\le k\le N}\, (\varphi_{_E} \circ T_{_E}^{k-1}) (x)
\end{equation} 
where we are trimming $\tau_{_{E,x}}$ by deleting the largest entry.

\begin{lemma}\label{useful-1}
For $\mu$-a.e.\ $x\in E$ it eventually holds
\[
N < \tau^1_{_{E,x}}(R_{_{E,N,m}}(x)) \le N + M^2_{_{N,m}}(x)
\]
where $M^2_{_{N,m}}(x)$ denotes the 2-nd maximum in $\{ (\varphi_{_E} \circ T_{_E}^{k-1}) (x)\, :\, k=1,\dots, R_{_{E,N,m}}(x)\}$ which in case of equality might coincide with the maximum. 
\end{lemma}

\begin{proof} We first give some remarks on $\tau_{_{E,x}}$ and $R_{_{E,N}}(x)$. Given $N\ge 1$, $R_{_{E,N}}(x)$ is the number of visits to $E$ up to time $N$, hence $R_{_{E,N}}(x)\le N+1$ and $T_{_E}^{R_{_{E,N}}(x)-1}(x)$ is the last visit to $E$ for $x$ up to time $N$. It follows that
\[
\tau_{_{E,x}}(R_{_{E,N}}(x)) = \sum_{k=1}^{R_{_{E,N}}(x)}\, (\varphi_{_E} \circ T_{_E}^{k-1}) (x) > N
\]
as the $(R_{_{E,N}}(x))$-th return to $E$, equivalently the $(R_{_{E,N}}(x)+1)$-th visit to $E$, happens after time $N$. Moreover
\[
\sum_{k=1}^{R_{_{E,N}}(x)}\, (\varphi_{_E} \circ T_{_E}^{k-1}) (x) = \varphi_{_E} \circ T_{_E}^{R_{_{E,N}}(x)-1} (x) + \tau_{_{E,x}}(R_{_{E,N}}(x)-1) \le \varphi_{_E} \circ T_{_E}^{R_{_{E,N}}(x)-1} (x) + N
\] 
as the $(R_{_{E,N}}(x)-1)$-th return to $E$, equivalently the $(R_{_{E,N}}(x))$-th visit to $E$, happens at time less than or equal to $N$.

Using now $R_{_{E,N,m}}(x)$, we have proved 
\begin{equation}\label{ineq-fund}
N+m(N,E,x) < \tau_{_{E,x}}(R_{_{E,N,m}}(x)) \le N+m(N,E,x) + \varphi_{_E} \circ T_{_E}^{R_{_{E,N,m}}(x)-1} (x).
\end{equation}

We now show that
\begin{equation}\label{max-s1}
\max\left\{ (\varphi_{_E} \circ T_{_E}^{k-1}) (x)\, :\, k=1,\dots, R_{_{E,N,m}}(x)\right\} = \max\left\{ m(N,E,x),\,  \varphi_{_E} \circ T_{_E}^{R_{_{E,N,m}}(x)-1} (x)\right\}.
\end{equation}
From \eqref{relations}, $m(N,E,x)$ maximises the return times $(\varphi_{_E} \circ T_{_E}^{k-1}) (x)$ for $k=1,\dots, R_{_{E,N}}(x)$, that is up to time $\tau_{_{E,x}}(R_{_{E,N}}(x))$. Moreover, repeating the same argument as above, the return times $(\varphi_{_E} \circ T_{_E}^{k-1}) (x)$ with $k=R_{_{E,N}}(x)+1,\dots, R_{_{E,N,m}}(x)-1$ are concerned with excursions outside $E$ for the orbit of $x$ with time in $(N, N+m(N,E,x)]$. Then
\[
m(N,E,x) = \max\left\{ (\varphi_{_E} \circ T_{_E}^{k-1}) (x)\, :\, k=1,\dots, R_{_{E,N}}(x)-1\right\}
\]
and \eqref{max-s1} is proved.

Then, if
\[
\max\left\{ (\varphi_{_E} \circ T_{_E}^{k-1}) (x)\, :\, k=1,\dots, R_{_{E,N,m}}(x)\right\} = m(N,E,x)
\]
it immediately follows from \eqref{ineq-fund} that
\[
N < \tau_{_{E,x}}(R_{_{E,N,m}}(x)) - \max_{1\le k\le R_{_{E,N,m}}(x)}\, (\varphi_{_E} \circ T_{_E}^{k-1}) (x) = \tau^1_{_{E,x}}(R_{_{E,N,m}}(x))\, ,
\]
otherwise, we obtain
\[
\tau^1_{_{E,x}}(R_{_{E,N,m}}(x)) = \tau_{_{E,x}}(R_{_{E,N,m}}(x)-1) >N
\]
since $x$ necessarily visits $E$ at some time $k$ in $(N,N+m(N,E,x)]$. In both cases we have thus obtained the first inequality to prove.

Moreover from \eqref{ineq-fund} and \eqref{max-s1} we obtain
\[
\tau_{_{E,x}}(R_{_{E,N,m}}(x)) - \max_{1\le k\le R_{_{E,N,m}}(x)}\, (\varphi_{_E} \circ T_{_E}^{k-1}) (x) \le N+ \min\left\{ m(N,E,x),\,  \varphi_{_E} \circ T_{_E}^{R_{_{E,N,m}}(x)-1} (x)\right\}
\]
and by \eqref{max-s1}
\[
 \min\left\{ m(N,E,x),\,  \varphi_{_E} \circ T_{_E}^{R_{_{E,N,m}}(x)-1} (x)\right\} \le M^2_{_{N,m}}(x).
 \]
The proof is finished.
\end{proof}

\begin{lemma}\label{useful-2}
For $\mu$-a.e.\ $x\in E$ we have
\[
\lim_{N\to \infty}\, \frac{\tau^1_{_{E,x}}(N)}{d(N)} = 1
\]
where $d(n)$ is the asymptotic inverse function of the sequence $(\alpha(n))$ defined in \eqref{alfa}.
\end{lemma}

\begin{proof} 
By their definition \eqref{trimmed-returns}, the terms $\tau^1_{_{E,x}}(N)$ are the trimmed Birkhoff sums of $\varphi_{_E}$ for the system $(E,T_{_E})$. Thus, it is enough to show that Lemma \ref{lemma-aar-nakada} can be applied to the sequence of random variables $(\varphi_{_E}\circ T_{_E}^{n-1})$. 

By assumption (i), the sequence is $\psi$-mixing with mixing coefficient $\psi(n)$ fulfilling $\sum_{n\ge 1} \psi(n)/n <\infty$. Moreover, the distribution function of $(\varphi_{_E}\circ T_{_E}^{n-1})$ is given by
\[
F(y) = \mu (\varphi_{_E}\le y) = \sum_{k=1}^{\lfloor y\rfloor}\, \mu(A_k) = 1 - \mu(A_{>\lfloor y\rfloor}),
\]
by \eqref{level-sets} and \eqref{super-level-sets}. Thus, in \eqref{eq: cond AN} with $y_0=1$ we find
\[
\begin{aligned}
& \int_{1}^{\infty}\left(\frac{y\left(1-F(y)\right)}{\int_0^y\left(1-F(t)\right)\mathrm{d}t}\right)^{r+1} \frac{1}{y}\, \mathrm{d}y = \sum_{n\ge 1}\, \int_n^{n+1}\, \left( \frac{y\, \mu(A_{>n})}{\int_0^y\, \mu(A_{>\lfloor t\rfloor})\mathrm{d}t}\right)^{r+1} \frac 1y\, \mathrm{d}y \le \\[0.2cm]
& \le \sum_{n\ge 1}\, \frac{(n+1)^r\, (\mu(A_{>n}))^{r+1}}{(\int_0^n\, \mu(A_{>\lfloor t\rfloor})\mathrm{d}t)^{r+1}} = \sum_{n\ge 1}\, \frac{(n+1)^r\, (\mu(A_{>n}))^{r+1}}{(\sum_{j=0}^{n-1}\, \mu(A_{>j}))^{r+1}} 
\end{aligned}
\]
and by assumption (ii) we find $W\le 1$. Analogously,
\[
\int_{1}^{\infty}\left(\frac{y\left(1-F(y)\right)}{\int_0^y\left(1-F(t)\right)\mathrm{d}t}\right)^{r+1} \frac{1}{y}\, \mathrm{d}y \ge 
\sum_{n\ge 1}\, \frac{n^r\, (\mu(A_{>n+1}))^{r+1}}{(\sum_{j=0}^{n}\, \mu(A_{>j}))^{r+1}} 
\]
and \eqref{somma-inf} imply $W>0$, so that $W=1$. Thus, the proof is finished by applying Lemma \ref{lemma-aar-nakada} and recalling that the Birkhoff sums of $(\varphi_{_E}\circ T_{_E}^{n-1})$ trimmed by the largest entry are $\tau^1_{_{E,x}}(N)$.
\end{proof}

\begin{lemma}\label{useful-3}
For $\mu$-a.e.\ $x\in E$ it holds
\[
\lim_{N\to \infty}\, \frac{R_{_{E,N,m}}(x)}{\alpha(N)} = 1\, .
\]
\end{lemma}

\begin{proof} Since $R_{_{E,N,m}}(x)$ is diverging, we obtain from Lemma \ref{useful-2} that
\[
\lim_{N\to \infty}\, \frac{\tau^1_{_{E,x}}(R_{_{E,N,m}}(x))}{d(R_{_{E,N,m}}(x))} = 1
\]
for $\mu$-a.e.\ $x\in E$. Moreover, from Lemma \ref{useful-1} we obtain
\[
\frac{N}{d(R_{_{E,N,m}}(x))} \le  \frac{\tau^1_{_{E,x}}(R_{_{E,N,m}}(x))}{d(R_{_{E,N,m}}(x))} \le \frac{N + M^2_{_{N,m}}(x)}{d(R_{_{E,N,m}}(x))} 
\]
where we recall that $M^2_{_{N,m}}(x)$ denotes the 2-nd maximum in $\{ (\varphi_{_E} \circ T_{_E}^{k-1}) (x)\, :\, k=1,\dots, R_{_{E,N,m}}(x)\}$. Thus, by applying Lemma \ref{lemma-aar-nakada} to $(\varphi_{_E}\circ T_{_E}^{n-1})$ with $W=1$ as shown above, we use \eqref{eq: M/d_n} with $r=2$ to obtain
\[
\lim_{N\to \infty}\, \frac{N}{d(R_{_{E,N,m}}(x))}  =\lim_{N\to \infty}\, \frac{N+M^2_{_{N,m}}(x)}{d(R_{_{E,N,m}}(x))} = 1
\]
for $\mu$-a.e.\ $x\in E$. The result now follows by using that $\alpha(n)$ is the asymptotic inverse of $d(n)$.
\end{proof}

Theorem \ref{th-elle1} follows by applying Lemma \ref{lem-hopf} and Lemma \ref{useful-3}. \qed

\subsection{Proof of Theorem \ref{th-elle1-bis}} As before we prove the theorem for $\mu$-a.e.\ $x\in E$. The first step is the analog of Lemma \ref{lem-hopf}. Using Hopf's Ratio Ergodic Theorem we have for $\mu$-a.e.\ $x\in E$
\begin{equation}\label{lem-hopf-2}
\lim_{N\to \infty}\, \frac{\sum_{k=1}^{N}\, (f\circ T^{k-1})(x) }{R_{_{E,N}}(x)} =  \lim_{N\to \infty}\, \frac{\sum_{k=1}^{N}\, (f\circ T^{k-1})(x) }{(\UU_{_E}\circ T^{N})(x)+\sum_{k=1}^{N}\, (\UU_{_E}\circ T^{k-1})(x)} = \int_X\, f\, \rd\mu.
\end{equation}
Then we consider the sequence $\tau_{_{E,x}}$ defined in \eqref{returns} and the trimmed version $\tau^1_{_{E,x}}$. In Lemma \ref{useful-1} we have proved
\begin{equation} \label{stima-tau}
N< \tau_{_{E,x}}(R_{_{E,N}}(x)) \le N + \varphi_{_E} \circ T_{_E}^{R_{_{E,N}}(x)-1}(x)
\end{equation}
for $\mu$-a.e.\ $x\in E$. For the trimmed sums 
\[
\tau^1_{_{E,x}}(R_{_{E,N}}(x)) = \tau_{_{E,x}}(R_{_{E,N}}(x)) - \max_{1\le k\le R_{_{E,N}}(x)}\, (\varphi_{_E} \circ T_{_E}^{k-1}) (x)  
\]
we now show that
\begin{equation} \label{useful-1-bis}
N-w(N,E,x) \le \tau^1_{_{E,x}}(R_{_{E,N}}(x)) \le N -w(N,E,x) + M^2_{_N}(x)
\end{equation}
where $w(N,E,x)$ is defined in \eqref{def-w} and $M^2_{_{N}}(x)$ denotes the 2-nd maximum in $\{ (\varphi_{_E} \circ T_{_E}^{k-1}) (x)\, :\, k=1,\dots, R_{_{E,N}}(x)\}$.

We remark that analogously to \eqref{relations} we have for $\mu$-a.e.\ $x\in E$
\begin{equation}\label{relations-w}
w(N,E,x) = \max \Big\{ \set{(\varphi_{_E}\circ T_{_E}^{k-1})(x)\, :\, k=1,\dots, R_{_{E,N}}(x)-1}\, ,\, N- \tau_{_{E,x}}(R_{_{E,N}}(x)-1) \Big\}
\end{equation}
since $ \tau_{_{E,x}}(R_{_{E,N}}(x)-1)$ is the time of the $(R_{_{E,N}}(x)-1)$-th return to $E$, the last before step $N+1$. We thus consider two cases. If
\[
\max\left\{ (\varphi_{_E} \circ T_{_E}^{k-1}) (x)\, :\, k=1,\dots, R_{_{E,N}}(x)\right\} = \max\left\{ (\varphi_{_E} \circ T_{_E}^{k-1}) (x)\, :\, k=1,\dots, R_{_{E,N}}(x)-1\right\},
\]
by \eqref{relations-w} it follows
\[
w(N,E,x) = \max\left\{ (\varphi_{_E} \circ T_{_E}^{k-1}) (x)\, :\, k=1,\dots, R_{_{E,N}}(x)-1\right\}
\]
since $\varphi_{_E} \circ T_{_E}^{R_{_{E,N}}(x)-1}(x) > N- \tau_{_{E,x}}(R_{_{E,N}}(x)-1)$. So \eqref{useful-1-bis} follows from \eqref{stima-tau} by subtracting $w(N,E,x)$ from all terms and using $\varphi_{_E} \circ T_{_E}^{R_{_{E,N}}(x)-1}(x) \le M^2_{_N}(x)$. If on the contrary
\[
\varphi_{_E} \circ T_{_E}^{R_{_{E,N}}(x)-1}(x) = \max\left\{ (\varphi_{_E} \circ T_{_E}^{k-1}) (x)\, :\, k=1,\dots, R_{_{E,N}}(x)\right\},
\]
then 
\[
\tau^1_{_{E,x}}(R_{_{E,N}}(x)) = \tau_{_{E,x}}(R_{_{E,N}}(x)) - \varphi_{_E} \circ T_{_E}^{R_{_{E,N}}(x)-1}(x) =\tau_{_{E,x}}(R_{_{E,N}}(x)-1).
\]
Now from \eqref{relations-w} we may have
\[
w(N,E,x) = N- \tau_{_{E,x}}(R_{_{E,N}}(x)-1) = N- \tau^1_{_{E,x}}(R_{_{E,N}}(x))
\]
in which case \eqref{useful-1-bis} is immediate. Otherwise we may have
\[
w(N,E,x) = \max\left\{ (\varphi_{_E} \circ T_{_E}^{k-1}) (x)\, :\, k=1,\dots, R_{_{E,N}}(x)-1\right\} 
\]
in which case
\[
N- \tau^1_{_{E,x}}(R_{_{E,N}}(x)) \le w(N,E,x) = M^2_{_N}(x)
\]
so that
\[
N- w(N,E,x) \le  \tau^1_{_{E,x}}(R_{_{E,N}}(x)) \le N = N - w(N,E,x) + M^2_{_N}(x)
\]
where we have used \eqref{stima-tau} in the second inequality, and again we have proved \eqref{useful-1-bis}.

Finally, we use Lemma \ref{useful-2} to write as in the proof of Lemma \ref{useful-3}
\[
\lim_{N\to \infty}\, \frac{\tau^1_{_{E,x}}(R_{_{E,N}}(x))}{d(R_{_{E,N}}(x))} = 1
\]
for $\mu$-a.e.\ $x\in E$. Moreover, from \eqref{useful-1-bis} we obtain
\[
\frac{N-w(N,E,x)}{d(R_{_{E,N}}(x))} \le  \frac{\tau^1_{_{E,x}}(R_{_{E,N}}(x))}{d(R_{_{E,N}}(x))} \le \frac{N -w(N,E,x) + M^2_{_{N}}(x)}{d(R_{_{E,N}}(x))}. 
\]
Thus, by applying Lemma \ref{lemma-aar-nakada} to $(\varphi_{_E}\circ T_{_E}^{n-1})$ with $W=1$ as above, we use \eqref{eq: M/d_n} with $r=2$ to obtain
\[
\lim_{N\to \infty}\, \frac{N-w(N,E,x)}{d(R_{_{E,N}}(x))} = 1
\]
for $\mu$-a.e.\ $x\in E$. Using now that $\alpha(n)$ is the asymptotic inverse of $d(n)$, we obtain
\begin{equation}\label{fine-2}
\lim_{N\to \infty}\, \frac{R_{_{E,N}}(x)}{\alpha(N-w(N,E,x))} = 1
\end{equation}
for $\mu$-a.e.\ $x\in E$. Theorem \ref{th-elle1-bis} follows from \eqref{lem-hopf-2} and \eqref{fine-2}.
\qed

\subsection{Proof of Lemma \ref{lem-bound-wt}}
We first prove that
 \begin{align}
  \mu\left(\left\{x\in E\, :\, \max_{1\le i\le n} \left(\varphi_{_E}\circ T_{_E}^{i-1}\right)(x)>q\left(n\, \log n \, \log^2(\log n) \right)\text{for infinite $n$}\right\}\right)=0.\label{eq: max not too large}
 \end{align}
We note that for all $2^n\leq k< 2^{n+1}$ 
\begin{equation}\label{eq: BC 1}
 \begin{aligned}
& \left\{x\in E\, :\, \max_{1\le i\le k} \left(\varphi_{_E}\circ T_{_E}^{i-1}\right)(x)>q\left( k\, \log k \, \log^2(\log k) \right)\right\}\, \subset \\[0.2cm]
  &\subset \left\{x\in E\, :\, \max_{1\le i\le 2^{n+1}} \left(\varphi_{_E}\circ T_{_E}^{i-1}\right)(x)>q\left( 2^n\, \log 2^n \, \log^2(\log 2^n) \right)\right\}.
 \end{aligned}
\end{equation}
 Furthermore, we have by the definition of $q$ that 
 \begin{equation}
 \begin{aligned} \label{eq: BC 2}
& \mu\left(\left\{x\in E\, :\, \max_{1\le i\le 2^{n+1}} \left(\varphi_{_E}\circ T_{_E}^{i-1}\right)(x)>q\left(  2^n\, \log 2^n \, \log^2(\log 2^n)\right)\right\} \right)\le\\[0.2cm]
  &\le  \sum_{j=1}^{2^{n+1}}\mu\left(\left\{x\in E\, :\, \left(\varphi_{_E}\circ T_{_E}^{j-1}\right)(x)>q\left( 2^n\, \log 2^n \, \log^2(\log 2^n)\right) \right\}\right)=\\[0.2cm]
  &= \sum_{j=1}^{2^{n+1}}\frac{1}{2^n\, \log 2^n \, \log^2(\log 2^n)} =\frac{2}{\log 2^n \, \log^2(\log 2^n)}.
 \end{aligned}
 \end{equation}
 Since
 \begin{align*}
  \sum_{n=1}^{\infty} \frac{2}{\log 2^n \, \log^2(\log 2^n)}<\infty,
 \end{align*}
 we obtain by the first Borel-Cantelli lemma and a combination of \eqref{eq: BC 1} and \eqref{eq: BC 2} that \eqref{eq: max not too large} holds. From \eqref{fine-2} we obtain that for $\mu$-a.e.\ $x\in E$  
 \begin{align*}
  \limsup_{n\to\infty}\frac{R_{_{E,n}}(x)}{\alpha(n)}\leq 1.
 \end{align*}
 Hence,
 \[
    \mu\left(\left\{ x\in E\, :\, \max_{1\le i\le R_{_{E,n}}(x)} \left(\varphi_{_E}\circ T_{_E}^{i-1}\right)(x)>q\left( \alpha(n)\, \log (\alpha(n)) \, \log^2(\log (\alpha(n))) \right)\text{for infinite $n$}\right\}\right)=0
 \]
and with the definition of $\xi$ the statement of the lemma follows. 
\qed

\subsection{Proof of Theorem \ref{th-notelle1}} 
For simplicity let us first assume that $f$ is constant on the level sets $E_n$ and we can take $g_1= g_2=f\ge 0$ in the assumptions. Hence there exists a non-decreasing $G:\N \to \R_{\ge 0}$ with $f^E(x) = G(n)$ for all $x\in A_n$, such that $G(n)= n\, L(n)$ with $L$ a normalised slowly varying function and that satisfies (ii)-(c), (d) and (e) with $\Gamma_i(n)=\Gamma(n) := \min\{ k\in \N\, :\, G(k)> n\}$. In this case the sequence of random variables $(f^E\circ T_{_E}^{n-1})$ is $\psi$-mixing since it generates the same $\sigma$-field as the sequence $(\varphi_{_E}\circ T_{_E}^{n-1})$. Thus it satisfies the first assumption of Lemma \ref{lemma-aar-nakada}. 

Moreover the distribution function of the random variables $(f^E\circ T_{_E}^{n-1})$ is given by
\[
F(y) = \mu ( f^E \le y) = 1 - \mu (A_{\ge \Gamma(y)})
\]
so that \eqref{cond-fe} implies $W\le 1$ in \eqref{eq: cond AN}. Setting $G(0)=0$ we have by definition $\Gamma(y) = k$ for all $y\in [G(k-1),G(k))$ for all $k\ge 1$. Hence
\[
\int_0^\infty\, \mu (A_{\ge \Gamma(y)}) \mathrm{d}y = \sum_{k=1}^\infty\, (G(k)-G(k-1))\, \mu(A_{> k-1}) = \sum_{k=1}^\infty\, f_{k-1}\, \mu(A_{> k-1}) = \int_0^\infty\, f(x) \mathrm{d}\mu = \infty
\]
where we have used \eqref{fe-const} with $f_n= f|_{E_n}$ for all $n\ge 0$, with $E_n$ the level sets \eqref{level-sets-hit} of the hitting time function to $E$. It follows that $W>0$, so $W=1$.

Applying Lemma \ref{lemma-aar-nakada} to $(f^E\circ T_{_E}^{n-1})$ it follows that letting $d(n)$ be the inverse function of 
\begin{equation} \label{alfa-new}
a(y) := \frac{y}{ \int_0^y\, \mu (A_{\ge \Gamma(t)}) \mathrm{d}t}
\end{equation}
we have for $N\to \infty$
\begin{equation}\label{limit-fe}
\sum_{n=1}^{R_{_{E,N}}(x) -1}\, (f^E \circ T^{n-1}_{_E})(x) \sim d(R_{_{E,N}}(x) -1) + \max_{1\le k\le R_{_{E,N}}(x) -1}\, (f^E \circ T^{k-1}_{_E})(x)
\end{equation}
for $\mu$-a.e.\ $x\in E$, since $R_{_{E,N}}(x)$ is diverging.

Looking at \eqref{f-vs-fe} we need now to consider the last term on the right hand side. We recall that if $\tau_{_{E,x}}(R_{_{E,N}}(x)-1)=N$ this term does not appear. If $\tau_{_{E,x}}(R_{_{E,N}}(x)-1)<N$ let us first write
\[
\sum_{k=\tau_{_{E,x}}(R_{_{E,N}}(x)-1)+1}^{N}\, (f\circ T^{k-1})(x) = \sum_{k=\tau_{_{E,x}}(R_{_{E,N}}(x)-1)+1}^{\tau_{_{E,x}}(R_{_{E,N}}(x))}\, (f\circ T^{k-1})(x) - \sum_{k=N+1}^{\tau_{_{E,x}}(R_{_{E,N}}(x))}\, (f\circ T^{k-1})(x)
\]
where $\tau_{_{E,x}}(R_{_{E,N}}(x))\ge N+1$ by definition. Then by definition \eqref{induced-obs} and by \eqref{fe-const}
\[
\Big(f^E \circ T_{_E}^{R_{_{E,N}}(x) -1}\Big)(x) = \sum_{k=\tau_{_{E,x}}(R_{_{E,N}}(x)-1)+1}^{\tau_{_{E,x}}(R_{_{E,N}}(x))}\, (f\circ T^{k-1})(x) =
\]
\[
=  \sum_{j=0}^{(\varphi_{_E}\circ T_{_E}^{R_{_{E,N}}(x) -1})(x)-1}\, f_j = G\Big(\Big(\varphi_{_E}\circ T_{_E}^{R_{_{E,N}}(x) -1}\Big)(x)\Big)
\]
and by \eqref{fe-const}
\[
\sum_{k=N+1}^{\tau_{_{E,x}}(R_{_{E,N}}(x))}\, (f\circ T^{k-1})(x) = \sum_{j=1}^{\tau_{_{E,x}}(R_{_{E,N}}(x))-N}\, f_j = G\Big(\tau_{_{E,x}}(R_{_{E,N}}(x))-N+1\Big) - f_0\, .
\]
Since $(f^E \circ T^{k-1}_{_E})(x) = G((\varphi_{_E}  \circ T^{k-1}_{_E})(x))$, using \eqref{limit-fe} in \eqref{f-vs-fe}, we finally have
\begin{equation}\label{first-step}
\begin{aligned}
S_Nf(x) \sim &\,  d(R_{_{E,N}}(x) -1) + \max_{1\le k\le R_{_{E,N}}(x) -1}\, G\Big((\varphi_{_E}  \circ T^{k-1}_{_E})(x)\Big) +\\[0.2cm]
& + G\Big(\Big(\varphi_{_E}\circ T_{_E}^{R_{_{E,N}}(x) -1}\Big)(x)\Big) - G\Big(\tau_{_{E,x}}(R_{_{E,N}}(x))-N+1\Big) 
\end{aligned}
\end{equation}
where we have discarded the constant $f_0$.

Let us fix $x\in E$, a constant $\eta\in (0,\frac 12)$\footnote{The bound $\eta<\frac 12$ simplifies a step of the proof but is not necessary.}, and let $(U_j)$ be the sequence of the return times to $E$, that is $U_j = \tau_{_{E,x}}(j)$ for $j\ge 1$ and $R_{_{E,U_j}}(x) -1=j$. Then we consider the following subintervals of $\N$ depending on $x$:
\begin{itemize}
\item the intervals corresponding to excursions outside $E$ which do not achieve a new record for the return time
\begin{equation}\label{int-min}
\II_j := \left\{ N\in (U_j, U_{j+1})\, :\, (\varphi_{_E}  \circ T^{j}_{_E})(x) \le \max_{1\le k\le j}\, (\varphi_{_E}  \circ T^{k-1}_{_E})(x)\right\}\, ;
\end{equation}

\item the intervals corresponding to excursions outside $E$ which achieve a new record for the return time
\begin{equation}\label{int-max}
\JJ_j := \left\{ N\in (U_j, U_{j+1})\, :\, (\varphi_{_E}  \circ T^{j}_{_E})(x) > \max_{1\le k\le j}\, (\varphi_{_E}  \circ T^{k-1}_{_E})(x)\right\}\, ;
\end{equation}
we also need to consider the subintervals for which $N$ is asymptotically not too far from $U_{j+1}$ with respect to $U_j$ up to $\eta$ in the following sense
\begin{equation}\label{int-max-eta}
\JJ_j^\eta := \left\{ N\in \JJ_j\, :\, U_{j+1}-N \le \eta\, (N-U_{j})  \right\}\, .
\end{equation}
\end{itemize}
The sequence $(U_j)$ and the intervals $\II_j$ and $\JJ_j$ depend on $x$, we have dropped this dependence in the notation for simplicity.

We now show that for $\mu$-a.e.\ $x\in E$ we have
\begin{equation} \label{final-1}
\lim_{N\to \infty}\, \frac{1}{G(N)}\, \sum_{k=1}^{N}\, (f\circ T^{k-1})(x) = 1\, .
\end{equation}
The proof follows from the following lemmas.

\begin{lemma} \label{lemma-step1}
For $\mu$-a.e.\ $x\in X$ we have
\begin{equation} \label{need-asymp}
d(R_{_{E,N}}(x) -1) \sim d(R_{_{E,N}}(x)) \sim G(N-w(N,E,x))\, .
\end{equation}
\end{lemma}

\begin{proof}
Here we use the notions of slowly and regularly varying functions from Appendix \ref{sec:svf}. By a remark in \cite{aar-nakada}, condition \eqref{cond-fe} implies that the sequence $(a(n))$ defined in \eqref{alfa-new} is a regularly varying function with exponent $1$, and the same is true for its inverse $d(n)$ (see for example \cite[Prop. 1.5.14]{reg-var-book}). For such a function we have that $d(b_n)\sim d(b_n-1)$ for all diverging sequences of positive numbers $(b_n)$. This shows that for $\mu$-a.e.\ $x\in X$ we have $d(R_{_{E,N}}(x) -1) \sim d(R_{_{E,N}}(x))$.

Moreover, from \eqref{fine-2} and again from $d$ being regularly varying with index 1, we have $d(R_{_{E,N}}(x)) \sim d(\alpha(N-w(N,E,x)))$, where $\alpha(n)$ is given in \eqref{alfa}. Using now the theory of the \emph{de Bruijn conjugate} of a slowly varying function (see \cite[Thm. 1.5.13]{reg-var-book}), letting $a(n) = n/ \ell(n)$ with 
\[
\ell(n):= \sum_{k=0}^{n}\, \mu(A_{\ge \Gamma(k)})\, ,
\]  
by assumption (ii)-(e) we know that $\ell$ is super-slowly varying at infinity with rate function itself (see Definition \ref{def-ssv}), and by \cite[Cor. 2.3.4]{reg-var-book} it follows that its de Bruijn conjugate $\ell^{\#}$ satisfies $\ell^{\#}(n) \sim 1/\ell(n)$. Hence
\[
d(n) \sim n\, \ell(n)
\]
and by \eqref{cond-uffa}
\[
d(\alpha(n)) \sim \alpha(n)\, \ell\Big(\alpha(n)\Big) \sim G(n)
\]
as $n\to \infty$. It follows that
\[
 d\left(\alpha(N-w(N,E,x))\right)\sim G\left(N-w(N,E,x)\right)
\]
and the lemma is proved.
\end{proof}

\begin{remark}\label{wand-rate-slowly-var}
Here we show that under condition \eqref{cond-fe} also the sequence $(\alpha(n))$ defined in \eqref{alfa} is a regularly varying function with exponent $1$. Indeed by the remark in \cite{aar-nakada}, condition \eqref{cond-fe} implies that $\gamma(t)=\sum_{k=1}^{\infty} \mu\left(A_{k}\right)\, \min\left\{k \, L(k), t\right\}$ is normalised slowly varying. (Indeed, the remark only states slow variation, but proves actually normalised slow variation.) Since $L(t)$ is normalised slowly varying and $nL(n)$ is non-decreasing, we can show that
\begin{align*}
 \gamma(t\, L(t))
 = \sum_{k=1}^{\infty} \mu\left(A_{k}\right)\, \min\left\{k\, L(k), t\, L(t)\right\}
 = \sum_{k=1}^{\lfloor t\rfloor} \mu(A_{>k})\, \Big(k\,L(k)-(k-1)L(k-1)\Big)
\end{align*}
is normalised slowly varying as well. Referring to \eqref{eq:karamata-norm}, we may write $\gamma(t)=\exp(\int_{\kappa}^t (\eta(x)/x) \mathrm{d}x)$ with $\eta(x)$ tending to zero and $L(t)=\exp(\int_{\kappa}^t (\varepsilon(x)/x) \mathrm{d}x)$ with $\varepsilon(x)$ tending to zero, by extending the domain of definition of $L$ by linear interpolation. Then we have for $t\notin\mathbb{N}$
\begin{equation} \label{der-gamma}
 \frac{\mathrm{d}}{\mathrm{d}t} \gamma(t\, L(t))= \frac{ \gamma(t\, L(t))\, \eta(t\, L(t))\, \left(1+\varepsilon(t)\right)}{t}
\end{equation}
and since $\eta(t\, L(t))\, \left(1+\varepsilon(t)\right)$ tends to zero, $\gamma(t\, L(t))$ is normalised slowly varying (see eq.\ (1.3.4) and subsequent comments in \cite{reg-var-book}).

In the next steps we will compare $\gamma(t\, L(t))$ with the wandering rate $w_n(E):=\sum_{k=1}^{n-1}\mu(A_{>k})$. Letting $\widetilde{L}(k):=1/(k\,L(k)-(k-1)L(k-1))$, we have  
\begin{align*}
 w_n(E)= & \sum_{k=1}^{n-1}\, \mu(A_{>k})\, (k\, L(k)-(k-1)L(k-1))\, \frac{1}{k\, L(k)-(k-1)L(k-1)}= \\[0.2cm]
 = & \sum_{k=1}^{n-1}\, \Big( \gamma(k\, L(k)) - \gamma((k-1)L(k-1))\Big)\, \widetilde{L}(k) = \sum_{k=1}^{n-1}\, \int_{k-1}^k\, \widetilde{L}(k) \, \left( \frac{\mathrm{d}}{\mathrm{d}t} \gamma(t\, L(t)) \right)\, \mathrm{d}t\, .
\end{align*}
Since $\widetilde{L}(k)$ is asymptotic to $1/L(k)$ it is slowly varying, then by \eqref{der-gamma} and \cite[Prop. 1.5.9a]{reg-var-book} also $w_n(E)$ is slowly varying. Thus, $\alpha(n)$ is a regularly varying function with exponent $1$.
\end{remark}

\begin{lemma}\label{lemma-2}
The limit in \eqref{final-1} holds for the sequence $(U_j)$.
\end{lemma}

\begin{proof}
For the sequence $(U_j)$, by \eqref{limit-fe} and \eqref{need-asymp} we first write
\[
S_{U_j}f(x) = \sum_{n=1}^{j}\, (f^E \circ T^{n-1}_{_E})(x) \sim G(U_j-w(U_j,E,x)) + \max_{1\le k\le j}\, G\Big((\varphi_{_E}  \circ T^{k-1}_{_E})(x)\Big)\, .
\]
By \eqref{relations-w}
\[
w(U_j,E,x) = \max_{1\le k\le j}\, (\varphi_{_E}  \circ T^{k-1}_{_E})(x)
\]
and since $G(n)$ is non-decreasing it follows
\[
S_{U_j}f(x) \sim G(U_j-w(U_j,E,x)) + G(w(U_j,E,x)) \sim G(U_j)
\]
by Lemma \ref{lemma-somme}.
\end{proof}

\begin{lemma}\label{lemma-3}
The limit in \eqref{final-1} holds for a sequence $(N_j)$ such that $N_j \in \II_j$ for all $j$.
\end{lemma}

\begin{proof}
By \eqref{int-min} and \eqref{relations-w}, for all $j$ it holds
\[
w(N_j,E,x) = \max_{1\le k\le R_{_{E,N_j}}(x) -1}\, (\varphi_{_E}  \circ T^{k-1}_{_E})(x)\, .
\]
Hence by \eqref{first-step} and \eqref{need-asymp} we have, since $G(n)$ is non-decreasing, 
\[
\begin{aligned}
S_{N_j}f(x) \sim &\,  G(N_j-w(N_j,E,x)) + G(w(N_j,E,x)) +\\[0.2cm]
& + G\Big(\Big(\varphi_{_E}\circ T_{_E}^{R_{_{E,N_j}}(x) -1}\Big)(x)\Big) - G\Big(\tau_{_{E,x}}(R_{_{E,N_j}}(x))-N+1\Big) \, .
\end{aligned}
\]
Moreover, since $(f^E\circ T_{_E}^{n-1})$ satisfies the assumptions of Lemma \ref{lemma-aar-nakada} with $W=1$, we apply \eqref{eq: M/d_n} to obtain
\[
G\Big(\Big(\varphi_{_E}\circ T_{_E}^{R_{_{E,N_j}}(x) -1}\Big)(x)\Big) = o(d(R_{_{E,N_j}}(x)))\, ,
\]
because by \eqref{int-min}
\[
G\Big(\Big(\varphi_{_E}\circ T_{_E}^{R_{_{E,N_j}}(x) -1}\Big)(x)\Big) \le G\Big(\max_{1\le k\le R_{_{E,N_j}}(x) -1}\, (\varphi_{_E}  \circ T^{k-1}_{_E})(x)\Big)
\]
so that $G\Big(\Big(\varphi_{_E}\circ T_{_E}^{R_{_{E,N_j}}(x) -1}\Big)(x)\Big)$ is the 2-nd maximum in 
\[
\left\{ G\Big((\varphi_{_E}  \circ T^{k-1}_{_E})(x)\Big)\, :\, k=1,\dots, R_{_{E,N_j}}(x) \right\}\, .
\]
Since $G(n)$ is non-negative, by using \eqref{need-asymp} and Lemma \ref{lemma-somme} the statement is proved.
\end{proof}

\begin{lemma}\label{lemma-4}
The limit in \eqref{final-1} holds for a subsequence $(N_j)$ such that $N_j \in \JJ_j^\eta$ for all $j$.
\end{lemma}

\begin{proof}
Looking at \eqref{first-step}, we first consider the last terms on the right-hand side. By the definition of $(U_j)$ and \eqref{returns}
\begin{equation}\label{translation}
\begin{aligned}
& \Big(\varphi_{_E}\circ T_{_E}^{R_{_{E,N_j}}(x) -1}\Big)(x) = \Big(\varphi_{_E}\circ T_{_E}^{R_{_{E,U_j}}(x) -1}\Big)(x) = (\varphi_{_E}\circ T_{_E}^{j})(x) = U_{j+1}- U_j\, ,  \\[0.2cm]
& \tau_{_{E,x}}(R_{_{E,N_j}}(x))-N_j = \tau_{_{E,x}}(R_{_{E,U_j}}(x))-N_j = \tau_{_{E,x}}(j+1)-N_j = U_{j+1}-N_j
\end{aligned}
\end{equation}
and by \eqref{int-max}, the sequence $U_{j+1}-U_j$ is diverging. Moreover writing
\[
(\varphi_{_E}\circ T_{_E}^{j})(x) = \Big(U_{j+1}-N_j+1\Big) + \Big(N_j-U_j-1\Big)  
\]
and using that $N_j\in \JJ_j^\eta$, letting $a_j := U_{j+1}-N_j+1$ and $b_j := N_j-U_j-1$, we have that $b_j$ is diverging by \eqref{int-max-eta}. Then by Lemma \ref{lemma-somme} we have
\begin{equation} \label{split1}
G\Big(\Big(\varphi_{_E}\circ T_{_E}^{R_{_{E,N_j}}(x) -1}\Big)(x)\Big) \sim G\Big(U_{j+1}-N_j+1\Big) + G\Big(N_j-U_j-1\Big).
\end{equation}
Moreover, we use \eqref{split1} and the fact that $G(n)=nL(n)$ with $L$ slowly varying to prove that
\begin{equation}\label{split2}
G\Big(\Big(\varphi_{_E}\circ T_{_E}^{R_{_{E,N_j}}(x) -1}\Big)(x)\Big) - G\Big(U_{j+1}-N_j+1\Big) \sim G\Big(N_j-U_j-1\Big)\, .
\end{equation}
In fact, by \eqref{split1} for all $\epsilon >0$ there exists $\bar{j}\in \N$ such that for all $j\ge \bar{j}$ we have
\[
1-\epsilon \le \frac{G\Big(\Big(\varphi_{_E}\circ T_{_E}^{R_{_{E,N_j}}(x) -1}\Big)(x)\Big)}{G\Big(U_{j+1}-N_j+1\Big) + G\Big(N_j-U_j-1\Big)} \le 1+\epsilon.
\]
Since $N_j-U_j\ge 1$ and $\eta\in (0,\frac 12)$, for $N_j\in \JJ_j^\eta$ we have
\[
U_{j+1}-N_j+1 \le \frac 12 (N_j-U_j) +1 \le N_j-U_j-1\, ,
\]
hence $G(U_{j+1}-N_j+1)\le G(N_j-U_j-1)$ because $G$ is non-decreasing. It follows that for $j\ge \bar{j}$
\[
1-2\epsilon\le \frac{G\Big(\Big(\varphi_{_E}\circ T_{_E}^{R_{_{E,N_j}}(x) -1}\Big)(x)\Big) - G\Big(U_{j+1}-N_j+1\Big)}{G\Big(N_j-U_j-1\Big)} \le 1+2\epsilon\, .
\]
By the arbitrariness of $\epsilon$ we have proved \eqref{split2}.

Using \eqref{limit-fe} and Lemma \ref{lemma-2} for the first two terms on the right-hand side of \eqref{first-step}, and \eqref{split2} for the last difference of \eqref{first-step}, we have
\[
S_{N_j}f(x) \sim G(U_j) + G(N_j-U_j-1)\, .
\]
Finally Lemma \ref{lemma-somme} and the fact that $N_j-U_j$ is diverging imply $S_{N_j}f(x) \sim G(N_j)$ and the lemma is proved.
\end{proof}

\begin{lemma}\label{lemma-5}
The limit in \eqref{final-1} holds for a subsequence $(N_j)$ such that $N_j \in \JJ_j \setminus \JJ_j^\eta$ for all $j$.
\end{lemma}

\begin{proof}
As in the proof of Lemma \ref{lemma-4} we need to study the last two terms on the right-hand side of \eqref{first-step}, for which we recall \eqref{translation}. For the first two terms, Lemma \ref{lemma-2} shows that they are asymptotic to $G(U_j)$.

To estimate the third and fourth term of \eqref{first-step} we aim to apply Lemma \ref{lem: an bn asym1}, for which we use that $U_{j+1}-U_j$ tends to infinity. Moreover we now need to consider two cases. Let us introduce the sets
\[
\LL^0_j := \left\{ N\in \JJ_j \, :\, \lim_{j\to \infty}\, \frac{U_{j+1}-U_j}{U_j}= 0\right\}
\]
and
\[
\LL_j := \left\{ N\in \JJ_j \, :\, \exists\, \ell\in (0,\infty] \, \text{ s.t.}\, \lim_{j\to \infty}\, \frac{U_{j+1}-U_j}{U_j}= \ell \right\}\, .
\]
It is clear that for all sequences $(N_j)$ there exists a subsequence in $\LL^0_j$ or in $\LL_j$. Thus we restrict ourselves to the cases $N_j\in \LL^0_j$ for all $j$, or $N_j\in \LL_j$ for all $j$.

First let $N_j \in (\JJ_j \setminus \JJ_j^\eta)\cap \LL^0_j$ for all $j$.
Then $U_{j}\sim U_{j+1}$ and $N_j\in (U_j, U_{j+1})$ imply $G(U_j)\sim G(U_{j+1})\sim G(N_j)$ and Lemma \ref{lemma-2} and the fact that $S_{U_j}f(x)<S_{N_j}f(x)<S_{U_{j+1}}f(x)$ imply that the limit \eqref{final-1} follows for the sequences in $\LL^0_j$.

Let now $N_j \in (\JJ_j \setminus \JJ_j^\eta)\cap \LL_j$ for all $j$ and let $\ell\in (0,\infty)$ be the limit of $(U_{j+1}-U_j)/U_j$. Notice that in this case we also have $\lim_{j} (U_{j+1}-N_j) =\infty$. A similar argument works for the case $\ell=\infty$.
Letting $a_j := U_{j+1}-U_j$ and $b_j:= U_{j+1}-N_j+1$, we apply Lemma \ref{lem: an bn asym1} to write
\begin{equation}\label{serve-1}
G(U_{j+1}-U_j) - G(U_{j+1}-N_j+1) \sim (N_j-U_j-1)\, L(U_{j+1}-U_j)\, .
\end{equation}
To study the term $ L(U_{j+1}-U_j)$, first by Lemma \ref{lem-bound-wt} for $\mu$-a.e.\ $x\in E$ we have for $j$ large enough
\begin{equation}\label{new2}
U_{j+1}-U_j = \max_{1\le k \le R_{_{E,U_j}}(x)} \, (\varphi_{_E}\circ T_{_E}^{k-1})(x) \le U_j\, \xi(U_j)\, .
\end{equation}
Then we want to use assumption (ii)-(c) to obtain the asymptotic behaviour of $L(U_{j+1}-U_j)$. 

Since $N_j \in (\JJ_j \setminus \JJ_j^\eta)\cap \LL_j$, for all $\epsilon>0$ there exists $j$ big enough such that $U_{j+1}-U_j\ge (\ell -\epsilon) U_j$. From \eqref{new2} it follows that $\tilde \xi(U_j) := (U_{j+1}-U_j)/U_j$ satisfies the conditions of assumption (ii)-(c), hence $L(U_{j+1}-U_j) \sim L(U_j)$. From \eqref{serve-1} and Lemma \ref{lemma-2} we have
\begin{equation}\label{serve-2}
S_{N_j}f(x) \sim G(U_j) + (N_j-U_j-1)\, L(U_j)\, .
\end{equation}
At this point, if
\[
\liminf_{j\to \infty}\, \frac{N_{j}-U_j}{U_j}= 0\, ,
\]
for all $\epsilon>0$ and for $j$ big enough we have, up to the choice of a subsequence, $0< N_{j}-U_j <\epsilon U_j$, and thus
\[
G(U_j) + (N_j-U_j-1)\, L(U_j) \le G(U_j) + \epsilon\, U_j\, L(U_j) \le (1+\epsilon)\, G(N_j)\, ,
\]
and in the other direction
\[
G(U_j) \ge G\Big( \frac{1}{1+\epsilon}\, N_j\Big) \apprge \frac{1}{1+\epsilon}\, G(N_j)\, .
\]
Hence by the arbitrariness of $\epsilon$, we find $S_{N_j}f(x) \sim G(N_j)$ in \eqref{serve-2}.

Instead, if
\[
\liminf_{j\to \infty}\, \frac{N_{j}-U_j}{U_j}= \bar \ell >0\, ,
\]
then for all $\epsilon>0$ and for $j$ big enough we have $N_{j}-U_j \ge (\bar \ell -\epsilon) U_j$. Hence letting $\tilde \xi(U_j):= (N_{j}-U_j)/U_j$, we can apply assumption (ii)-(c) since
\[
\bar \ell -\epsilon \le \tilde \xi(U_j) \le \frac{U_{j+1}-U_j}{U_j} \le \xi(U_j)\, .
\]
Thus it follows that $L(N_{j}-U_j) \sim L(U_j)$, and in \eqref{serve-2} we find
\[
S_{N_j}f(x) \sim G(U_j) + (N_j-U_j-1)\, L(N_j-U_j) \sim G(U_j) + G(N_j-U_j) \sim G(N_j)
\]
by applying Lemma \ref{lemma-somme} and the fact that $(U_j)$ and $(N_j-U_j)$ are diverging sequences.

This concludes the proof of the lemma.
\end{proof}

From Lemmas \ref{lemma-2} - \ref{lemma-5} the limit in \eqref{final-1} holds and the theorem is proved for functions $f$ constant on $E_n$.

In the general case, for a function $f:X\to \R_{\ge 0}$ satisfying assumption (ii) we obtain that
\begin{equation} \label{final-gen}
N\, L_1(N) \sim S_N g_1(x) \le S_N f(x) \le S_N g_2(x) \sim N\, L_2(N)
\end{equation}
for $\mu$-a.e.\ $x$, since we can apply Lemmas \ref{lemma-2} - \ref{lemma-5} to $g_1,g_2$. Moreover, for $\mu$-a.e.\ $x\in A_n$ we have $g_1^E(n) \le f^E(x) \le g_2^E(n)$ for all $n\in \N$. Since $g_1^E(n) \sim g_2^E(n)$ by assumption (ii)-(b), the function $G(n)$ which is asymptotically equivalent to $f^E$ is well-defined and satisfies $G(n) \sim g_i^E(n)=nL_i(n)$ for $i=1,2$. Thus by \eqref{final-gen} we obtain $S_Nf(x) \sim G(N)$ for $\mu$-a.e.\ $x$ and the theorem is proved.
\qed

\section{Proofs of the statements from Section \ref{sec-examples}} \label{app-minor}

\noindent \textbf{Proof of Proposition \ref{prop-2dfarey}}. In the proof of \cite[Theorem 3.1]{mio} it is shown that $(E,S_{_E})$ is a fibred system with respect to the level sets of the return time function $\varphi_{_E}$, hence it is $\psi$-mixing. It remains to show that $\sum_{n\ge 1}\, \psi(n)/n <\infty$.

Using \eqref{stima-psi} we need to study the sequences $\sigma(k)$ and $\gamma(k)$. Condition (h4) is satisfied with $N=1$ and $U_1=E$, so that $\gamma(k) \equiv 0$. Moreover it is proved in \cite{mio} (see the proof of Proposition 3.9 and the proof that (h3) holds for $(E,S_{_E})$), that there exists a constant $C>0$ such that $\sigma(k) \le C\, d(k)$, where $d(k)$ is defined by
\[
d(k) = \sum_{j=0}^2
        \left(\abs{\frac{f_{k+j+2}}{f_{k+6}}-\frac{f_{k+j+1}}{f_{k+5}}}
          + \abs{\frac{f_{k+j+2}}{f_{k+6}}-\frac{f_{k+j}}{f_{k+4}}}\right)
\]
being $(f_k)$ the sequence recursively defined as
\[
f_0=0,\, f_1=1,\, f_2=0,\quad f_{k+3}=f_{k+2}+f_k,\, \,  \forall\, k\ge 0\,.
\]
Then given the distinct roots $\lambda, \mu, \bar\mu$ of the polynomial $p(t)=t^3-t^2-1$, with $\lambda >1$ and $|\mu|^2=|\bar \mu|^2= \lambda^{-1}<1$, there exist constants $c_1,c_2,c_3\in \C$ such that
\[
f_k = c_1\, \lambda^k + c_2\, \mu^k + c_3\, \bar\mu^k\, , \quad \forall\, k\ge 0\, .
\]
It follows that $d(k) = O(\lambda^{-k})$, so that $\sigma(k)=O(\lambda^{-k})$, and finally $\sum_{n\ge 1}\, \psi(n)/n <\infty$. \qed
   
\vskip 0.5cm
\noindent \textbf{Proof of Proposition \ref{notelle1-linear}}. We show that the assumptions of Theorem \ref{th-notelle1} are satisfied with $G(N)\sim N$ as $N\to \infty$. It is clear that assumptions (ii)-(a), (b) and (c) are satisfied by (c1). In addition, since $\Gamma_i(k)\sim k$ we can argue as in the proof of Lemma \ref{useful-2} to show that (c2) implies (ii)-(d).

Thus it remains to show that (ii)-(e) is satisfied. First, in this case we have $\alpha(n) \sim n/\ell(n)$, hence \eqref{cond-uffa} becomes $\ell(N/\ell(N))\sim \ell(N)$. Finally, using that $\ell(n)$ is an increasing sequence which diverges, for all $\delta\in [0,1]$ we have the inequalities
\[
\ell\Big(n\, \ell^\delta(n) \Big) \ge \ell(n)
\]
and
\[
\ell\left(\frac{n\, \ell^\delta(n)}{\ell(n\, \ell^\delta(n))} \right) \le \ell\left( \frac{n\, \ell(n)}{\ell(n)}\right) = \ell(n)\, .
\]
Thus by (c3)
\[
\ell(n)\le \ell\Big(n\, \ell^\delta(n) \Big) \sim \ell\left(\frac{n\, \ell^\delta(n)}{\ell(n\, \ell^\delta(n))} \right) \le \ell(n)
\]
and we have shown that
\begin{equation} \label{altro-num}
\lim_{N\to \infty}\, \frac{\ell(n\, \ell^\delta(n))}{\ell(n)} = 1
\end{equation}
for all $\delta\in [0,1]$. This can be easily made uniform in $\delta$, proving that $\ell$ is super-slowly varying at infinity with rate function itself (see Definition \ref{def-ssv}). However we need this property of $\ell$ only in the proof of Lemma \ref{lemma-step1} to apply \cite[Cor. 2.3.4]{reg-var-book}, for which we only need \eqref{altro-num} with $\delta=1$.

Thus all the assumptions of Theorem \ref{th-notelle1} are satisfied with $G(N) \sim g_i^E(N) \sim N$. \qed

\vskip 0.5cm
\noindent \textbf{Proof of Proposition \ref{notelle1-counterex}}. All assumptions of Theorem \ref{th-notelle1} except for (ii)-(e) are satisfied, hence we can make use of all the results leading to the proof of Theorem \ref{th-notelle1} except for  Lemma \ref{lemma-step1}. 
 
We will only show the divergence result of this proposition for the subsequence $(U_j)$ fulfiling $U_j=\tau_{_{E,x}}(j)$ for $j\ge 1$. By \eqref{f-vs-fe} and \eqref{limit-fe} we have for this subsequence  
 \begin{equation}\label{eq: notconvsum}
  \sum_{n=1}^{U_j}\, (f \circ T^{n-1})(x) \sim d(R_{_{E,U_j}}(x) -1) + \max_{1\le k\le R_{_{E,U_j}}(x) -1}\, (f^E \circ T^{k-1}_{_E})(x).
 \end{equation}
 To estimate the first summand we obtain by the same argumentation as in Lemma \ref{lemma-step1} that $d(R_{_{E,U_j}} (x))\sim d(\alpha(U_j -w(U_j, E, x)))$. We may use again $\mu(A_{\ge\Gamma_n})\sim \log n/n$ with which we calculate $a(n)\sim 2\,n/\log^2(n)$ with $a$ as in \eqref{alfa-new}. Hence, we obtain for its asymptotic inverse function $d(n)\sim n\, \log^2(n)/2$.
 Using then $\alpha(n)\sim n/\log n$ from above and \eqref{fine-2} implies  
 \begin{equation}\label{eq: notconvsum1}
  d(R_{_{E,U_j}} (x)-1)\sim \frac{1}{2}\,(U_j -w(U_j, E, x))\, \log(U_j -w(U_j, E, x)).
 \end{equation}

On the other hand, for the second summand of \eqref{eq: notconvsum} we obtain from 
\[
(f^{E}\circ T_{_E}^{k-1})(x) \sim (\varphi_{_E}\circ T_{_E}^{k-1})(x)\, \log (\varphi_{_E}\circ T_{_E}^{k-1})(x)\, ,
\]
with $\sim$ in the meaning that $(f^{E}\circ T_{_E}^{k-1})(x)$ large, and \eqref{relations-w} that 
 \begin{equation}\label{eq: notconvsum2}
  \max_{1\le k\le R_{_{E,U_j}}(x) -1}\, (f^E \circ T^{k-1}_{_E})(x)\sim  w(U_j,E,x)\, \log w(U_j,E,x).
 \end{equation}
From Theorem \ref{th-elle1-bis} and the fact that $\alpha$ is regularly varying with index $1$ we can conclude that there exist $u_1'<u_2'$ such that for $\mu$-a.e.\ $x\in X$ we have 
\[
\liminf_j w(U_j,E,x)/U_j<u_1' < u_2' <\limsup_j w(U_j,E,x)/U_j \, .
\] 
Indeed if the limit of $w(U_j,E,x)/U_j$ existed, the result of Theorem \ref{th-elle1-bis} would contradict \cite[Thm. 2.4.2]{aa-book}. Hence, \eqref{eq: notconvsum}, \eqref{eq: notconvsum1} and \eqref{eq: notconvsum2} together imply the statement of Proposition \ref{notelle1-counterex}.  \qed

\appendix

\section{Proof of Lemma \ref{lemma-aar-nakada}} \label{sec:lemma-aar-nakada}
 Lemma \ref{lemma-aar-nakada} is an equivalent formulation to \cite[Theorem 1.1]{aar-nakada}. Instead of the condition in \eqref{eq: cond AN} this theorem asks for the following conditions to hold:
 $J_r\coloneqq \sum_{n=1}^{\infty}\epsilon(n)^r/n<\infty$ with 
 $\epsilon(n)=n\, \left(\log^+L\right)'(n)$ and $L(t)=\mathbb{E}\left( \min\{t, Y_1\}\right)$.
 We have that
 \begin{align*}
  L(t)&=\int_0^t x\mathrm{d}F(x)+ t\, \left(1-F(t)\right)
  = t\, F(t)- \int_0^t F(x)\mathrm{d}x+ t\, \left(1-F(t)\right)
  = \int_0^t (1- F(x))\mathrm{d}x.
 \end{align*}
 Furthermore,
 \begin{align*}
  \left(\log^+L\right)'(t)= \frac{L'(t)}{L(t)}=\frac{1-F(t)}{\int_0^t (1- F(x))\mathrm{d}x}\quad\text{ and }\quad \epsilon(t)=\frac{t\, \left(1-F(t)\right)}{\int_0^t (1- F(x))\mathrm{d}x}.
 \end{align*}
 Thus, the condition $J_r\coloneqq \sum_{n=1}^{\infty}\epsilon(n)^r/n<\infty$ and \eqref{eq: cond AN} are equivalent. 
 
Moreover, in \cite{aar-nakada} the norming sequence $(d(n))$ is set to be the inverse of 
$a(t)=t/L(t)= t/\int_0^t (1-F(x))\mathrm{d}x$. Since by the remark before Theorem 1.1 in \cite{aar-nakada}, $L$ is slowly varying (see Appendix \ref{sec:svf}), $\lim_{t\to\infty} t/L(t)=\infty$. Moreover, one can easily verify that $a'(t)>0$ on an interval $[K,\infty)$ and thus one can consider $a(t)$ as an invertible function. 
 
Equation \eqref{eq: M/d_n} follows from Theorem 1.1.(ii) and its following remark in \cite{aar-nakada}.
\qed

\section{Slowly varying functions} \label{sec:svf}
In this section we collect the results on slowly varying functions used in the proof of Theorem \ref{th-notelle1}. We recall that a function $L :\R_{>0} \to \R$ is called \emph{regularly varying (at infinity) with index $\gamma\in \R$} if for all $r>0$ it holds
\[
\lim_{x\to \infty}\, \frac{L(rx)}{L(x)} = r^\gamma\, .
\]
If $\gamma=0$, $L$ is called a \emph{slowly varying (at infinity)}. We refer to \cite{reg-var-book} for more details. By Karamata's representation theorem (see for example \cite[Thm.~1.3.1]{reg-var-book}) we have that each slowly varying function $L$ can be written as 
\begin{align}
L(x)=c(x)\, \exp\left(\int_{\kappa}^x\frac{\eta(t)}{t}\mathrm{d}t\right),\label{eq:karamata}
\end{align}
where $c(x)$ tends to a constant $C$ as $x\to \infty$, $\kappa\geq 0$ and $\eta(x)$ tends to zero as $x\to \infty$, and a function $L$ is called \emph{normalized slowly varying} if it can be written as 
\begin{align}
 L(x)=C\, \exp\left(\int_{\kappa}^x\frac{\eta(t)}{t}\mathrm{d}t\right).\label{eq:karamata-norm}
\end{align}
It is immediately clear that for each slowly varying function $L$ there exists a normalised slowly varying function $\widetilde{L}$ such that $L(x)\sim \widetilde{L}(x)$. 

In the proof of our results we need \emph{Potter's bound}, see for example \cite[Thm. 1.5.6]{reg-var-book}. If $L$ is a slowly varying function at infinity, then for all constants $\delta>0$ and $A>1$ there exists $C=C\left(\delta, A\right)$ such that for all $x,y\geq C$ we have
 \begin{equation} \label{potter-bound}
\frac{L\left(x\right)}{L\left(y\right)} \le A\, \max\left\{ \left( \frac{x}{y} \right)^\delta\, , \left( \frac{y}{x} \right)^{\delta} \right\}.
 \end{equation}

\begin{lemma} \label{lemma-somme}
If $(a_n)$ and $(b_n)$ are two sequences of non-negative reals, at least one of them diverging, and $L$ is non-negative and slowly varying such that $(n\, L(n))$ is non-decreasing, then
\[
a_n\, L(a_n) + b_n\, L(b_n) \sim (a_n+b_n)\, L(a_n+b_n)\, .
\]
\end{lemma}

\begin{proof}
We begin writing
\begin{align*}
 a_n L(a_n)+ b_n L(b_n)
 = \left(a_n + b_n\right) L(a_n) - b_n\, \left(L(a_n)- L(b_n)\right).
\end{align*}
Without restriction of generality we assume that $a_n>b_n$ for all $n$. (Otherwise we could just define two new sequences $c_n:=\max\{a_n, b_n\}$ and $d_n:=\min\{a_n, b_n\}$ and continue with these sequences.)
By Potter's bound \eqref{potter-bound} we have for all $\epsilon>0$ that there exists $N\in\mathbb{N}$ such that for $n\geq N$ 
\[
L(a_n+b_n) \le \left(1+\epsilon\right) L(a_n)\, \left(\frac{a_n+b_n}{a_n}\right)^{\epsilon} \le (1+\epsilon)\, 2^\epsilon\, L(a_n)
\]
and
\[
L(a_n+b_n) \ge \left(1-\epsilon\right) L(a_n)\, \left(\frac{a_n}{a_n+b_n}\right)^{\epsilon}  \ge (1-\epsilon)\, 2^{-\epsilon}\, L(a_n)\, ,
\]
implying $L(a_n+b_n)\sim L(a_n)$ and thus $\left(a_n + b_n\right) L(a_n)\sim \left(a_n + b_n\right) L(a_n+b_n)$. 

Next, we have a closer look at $b_n\, \left(L(a_n)- L(b_n)\right)$ distinguishing two cases. For a given $\epsilon>0$ we first look at the subsequence $n_k$ fulfilling $b_{n_k}\leq \epsilon \, a_{n_k}$ for which we have 
\begin{align*}
 b_{n_k}\, \left(L(a_{n_k})- L(b_{n_k})\right)< \epsilon\, \left(a_{n_k} + b_{n_k}\right)  L(a_{n_k})\sim \epsilon\, \left(a_{n_k} + b_{n_k}\right)  L\left(a_{n_k}+b_{n_k}\right)\, .
\end{align*}
On the other hand, $n\, L(n)$ being non-decreasing and non-negative we have
\begin{align*}
 b_{n_k}\, \left(L(a_{n_k})- L(b_{n_k})\right)
 &\geq  - b_{n_k}\, L\left(b_{n_k}\right)
 \geq -\epsilon\, \left(a_{n_k}+b_{n_k}\right)\, L\left(\epsilon\, \left(a_{n_k}+b_{n_k}\right)\right) \\
&\sim -\epsilon\, \left(a_{n_k}+b_{n_k}\right)\, L\left(a_{n_k}+b_{n_k}\right).
\end{align*}
Next, we look at the subsequence fulfilling $b_{n_k}> \epsilon \, a_{n_k}$. Applying again \eqref{potter-bound}, we have that for $k$ big enough
\[
L(a_{n_k}+b_{n_k}) \le (1+\epsilon)\, L(b_{n_k})\, \left( \frac{a_{n_k}+b_{n_k}}{b_{n_k}} \right)^{\epsilon} \le \frac{(1+\epsilon)^{1+\epsilon}}{\epsilon^\epsilon}\, L(b_{n_k})
\]
and
\[
L(a_{n_k}+b_{n_k}) \ge \frac{1}{1+\epsilon}\, L(b_{n_k})\, \left( \frac{b_{n_k}}{a_{n_k}+b_{n_k}} \right)^{\epsilon} \ge \frac{\epsilon^\epsilon}{(1+\epsilon)^{1+\epsilon}}\, L(b_{n_k})\, .
\]
Since $L(a_{n_k}+b_{n_k}) \sim L(a_{n_k})$ we also have that for $k$ big enough
\[
\frac{1}{1+\epsilon} \le \frac{L(a_{n_k}+b_{n_k})}{ L(a_{n_k})} \le 1+\epsilon\, ,
\]
therefore there is a constant $c>0$ such that choosing $k$ big enough we have
\begin{align*}
b_{n_k}\, \left(L(a_{n_k})- L(b_{n_k})\right) & \le (a_{n_k}+b_{n_k}) \left( (1+\epsilon) L(a_{n_k}+b_{n_k}) - \frac{\epsilon^\epsilon}{(1+\epsilon)^{1+\epsilon}}\, L(a_{n_k}+b_{n_k}) \right)\\[0.2cm] 
& \le c\, \epsilon^{1/2}\, (a_{n_k}+b_{n_k})\, L(a_{n_k}+b_{n_k})
\end{align*}
for $\epsilon\in (0,1)$, and similarly
\begin{align*}
b_{n_k}\, \left(L(a_{n_k})- L(b_{n_k})\right) 
& \ge b_{n_k} \left( \frac{1}{1+\epsilon}\, L(a_{n_k}+b_{n_k}) - \frac{(1+\epsilon)^{1+\epsilon}}{\epsilon^\epsilon}\, L(a_{n_k}+b_{n_k}) \right) \\[0.2cm] 
& \ge - c\, \epsilon^{1/2}\, (a_{n_k}+b_{n_k})\, L(a_{n_k}+b_{n_k})
\end{align*}
for $\epsilon\in (0,1)$. 

Since $\epsilon$ was chosen arbitrarily, we have that 
\[
b_{n_k}\, \left(L(a_{n_k})- L(b_{n_k})\right)=o\left( (a_{n_k}+b_{n_k})\, L(a_{n_k}+b_{n_k})\right)\, ,
\]
and the lemma is proved.
\end{proof}

\begin{lemma}\label{lem: an bn asym1}
Let $(a_n)$ and $(b_n)$ be two sequences of non-negative reals with $a_n>b_n$ and $(a_n)$ tending to infinity. Further let $L$ be a non-negative normalized slowly varying function. Then 
 \begin{align}
  a_n L(a_n)- b_n L(b_n)\sim (a_n-b_n) L(a_n)\, .\label{eq: anbn2}
 \end{align}
\end{lemma}
\begin{proof}
First we look at a subsequence $n_j$ for which $b_{n_j}$ is bounded, in this case \eqref{eq: anbn2} immediately holds.
Hence, we may assume without loss of generality that also $b_n$ tends to infinity. We have
\begin{align*}
 a_n L(a_n)- b_n L(b_n)
 &= \left(a_n - b_n\right) L(a_n) + b_n\, \left(L(a_n)- L(b_n)\right)\\
 &= b_n\left(\frac{a_n}{b_n}-1\right)L(a_n)+b_n\, \left(1- \frac{L(b_n)}{L(a_n)}\right)\, L(a_n).
\end{align*}
Hence, we have to show that 
\begin{align*}
1- \frac{L(b_n)}{L(a_n)}=o\left(\frac{a_n}{b_n}-1\right).
\end{align*}
To do so, using \eqref{eq:karamata-norm} we note that 
\begin{align*}
 \frac{L(a_n)}{L(b_n)}
 &=\exp\left(\int_{b_n}^{a_n}\frac{\eta(t)}{t}\mathrm{d}t \right)
 \leq \exp\left(\sup_{y\in [b_n,a_n]}\eta(y) \int_{b_n}^{a_n}\frac{1}{t}\mathrm{d}t\right)
\end{align*}
and similarly
\begin{align*}
 \frac{L(a_n)}{L(b_n)}
 &=\exp\left(\int_{b_n}^{a_n}\frac{\eta(t)}{t}\mathrm{d}t \right)
 \geq \exp\left(\inf_{y\in [b_n,a_n]}\eta(y) \int_{b_n}^{a_n}\frac{1}{t}\mathrm{d}t\right).
\end{align*}
Since $(b_n)$ tends to infinity and $\eta$ tends to zero we can conclude that for all $\epsilon>0$ there exists $N\in\mathbb{N}$
such that for all $n\geq N$ we have 
\begin{align*}
\left(\frac{a_n}{b_n}\right)^{-\epsilon}
 &\leq \frac{L(a_n)}{L(b_n)}
 \leq \left(\frac{a_n}{b_n}\right)^{\epsilon}.
\end{align*}
Hence, for all $\epsilon>0$ there exists $N\in\mathbb{N}$ such that for all $n\geq N$ 
\begin{align*}
 \frac{\left|1- \frac{L(b_n)}{L(a_n)}\right|}{\left|\frac{a_n}{b_n}-1\right|}
 &\leq \frac{\left(\frac{a_n}{b_n}\right)^{\epsilon}-1}{\frac{a_n}{b_n}-1}.
\end{align*}
Since the function $g\colon (1,\infty)\to\mathbb{R}$ given by $(x^{\epsilon}-1)/(x-1)$ is monotonically decreasing and tends to $\epsilon$ for $x\to 1^+$, it follows that for all $\epsilon>0$ there exists $N\in\mathbb{N}$ such that for all $n\geq N$  
\begin{align*}
\left|1- \frac{L(b_n)}{L(a_n)}\right| \leq \epsilon\, \left|\frac{a_n}{b_n}-1\right|\, .
\end{align*}
Since $\epsilon$ was arbitrary, the statement is proven.
\end{proof}

\begin{lemma} \label{normslow}
 Let $L$ be a slowly varying function at infinity. Then there exist two normalised slowly varying functions $L^-$ and $L^+$ such that $L^-(n)\leq L(n)\leq L^+(n)$ for all $n\in \mathbb{N}$ and $L^-(n)\sim L(n)\sim L^+(n)$ as $n\to \infty$. 
\end{lemma}
\begin{proof}
 Let $L$ be represented as in \eqref{eq:karamata} for fixed $c(x)$, $\kappa$ and $\eta(x)$. Our approach is to find two normalised slowly varying functions $c^-, c^+$ fulfilling $c^-\leq c\leq c^+$ and $c^-(n)\sim c(n)\sim c^+(n)$. Then we may set 
 \[
 L^{\pm}(n)=c^{\pm}(n)\, \exp\left(\int_{\kappa}^n\eta(t)/t\mathrm{d}t\right)\, .
 \] 
Since a product of two normalised slowly varying functions is still normalised slowly varying, the functions $L^{\pm}$ fulfil all the required properties. We will in the following only give the construction of $c^-$ as the construction of $c^+$ follows analogously. 
 
We notice that a function $\ell$ is normalised slowly varying if it fulfils $\ell'(x)=o\left(\ell(x)/x\right)$ almost everywhere. This follows immediately by taking the derivative of \eqref{eq:karamata-norm}, but see also \cite[p.\ 15]{reg-var-book}. We construct a function $c^-$ on a line $[\gamma_1,\infty)$ which is continuous and piecewise differentiable. 

In order to define $c^-$ let $\Gamma\coloneqq \left\{n\in\mathbb{N}\colon c(n)<c(k)\text{ for all }k>n\right\}$ and let $(\gamma_n)$ be the ordered sequence of elements in $\Gamma$, that is we have $\gamma_1<\gamma_2<\ldots$ and $c(\gamma_1)<c(\gamma_2)<\ldots$. The set $\Gamma$ contains infinitely many elements unless $\min\set{c(n),C}=C$ for sufficiently large $n$ which is a trivial case in which we might simply set $c^-(n)=\min\set{c(n),C}$. In this case $(c^-)'(x)=0$ for $x$ sufficiently large and thus $(c^-)'(x)\leq (x\log x)^{-1}$ holds immediately.
 
 As we are only interested in the limit behaviour, it is sufficient to define $c^-$ at $\mathbb{R}_{\geq\gamma_1}$ and we set 
 $c^-(\gamma_i)=c(\gamma_i)$, for all $i\in\mathbb{N}$.
 For $x\in (\gamma_i, \gamma_i+1)$ we set 
 \begin{align*}
  c^-(x):=\max\left\{(\log\, x)^{1/2}+c(\gamma_i)-(\log\, \gamma_i)^{1/2} ,c(\gamma_{i+1})\right\}.
 \end{align*}
The function $c^-$ is then either piecewise constant and on those parts its derivative equals zero, or has derivative equal to $1/2 (\log\, x)^{-1/2}\, x^{-1}=o(x^{-1})$.  This is sufficient since in our case, $c^-$ is bounded by $C$. Thus, at those points at which $c^-$ is differentiable, it fulfils $\left(c^-\right)'(x)=o\left(c^-(x)/x\right)$ and there are only countably many points on which it is not differentiable.  On the other hand, since $(\log\, x)^{1/2}$ tends to infinity, $c^-$ tends to $C$.
\end{proof}

\begin{lemma}\label{remarkii}
Let $h$ be a slowly varying function with $h(t)\ge 1+h_0$ for all $t$ and for some $h_0>0$, and let $L$ be a slowly varying function written as in \eqref{eq:karamata} with $\eta(t) = o(1/\log \xi(t))$ as $t\to \infty$. Then for all functions $\tilde h$ which satisfy $c\le \tilde h(n)\le h(n)$ for some constant $c>0$ and all $n$, it holds $L(n\, \tilde h(n)) \sim L(n)$ as $n\to \infty$.
\end{lemma}

The proof of the lemma follows by applying results in the theory of functions with regular variations for which we first need to introduce some definitions.

\begin{definition}\label{def-ssv}
Given a function $h: \R_{>0} \to \R_{>0}$, a function $\ell:\R_{>0}\to \R$ is called \emph{super-slowly varying at infinity with rate function} $h$ if  
 \[
  \lim_{x\to\infty}\frac{\ell\left(x\, h(x)^{\delta}\right)}{\ell(x)}=1 \text{ uniformly in }0\leq \delta\leq 1.
 \]
\end{definition}

\begin{definition}\label{def-self-cont}
A function $\phi: \R_{>0} \to \R$ is called \emph{self-controlled} if there exist positive constants $\gamma,\Gamma, T$ such that $\gamma \phi(t)\le \phi(t+\delta\phi(t))\le \Gamma \phi(t)$, for all $t\ge T$ and $\delta\in[0,1]$.
\end{definition}

 \begin{lemma}[{\cite[Thm.~3.12.5]{reg-var-book}}]\label{lem:super-slow-rep}
 Let $h: \R_{>0} \to \R_{>0}$ be such that $\phi(t):=\log (h(e^t))$ is self-controlled and there exists $T>0$ such that $1/\phi$ is locally integrable over $[T, \infty)$. 
  Then $L:\R_{>0}\to \R$ is super-slowly varying with rate function $h$ if and only if 
  $L$ can be written as in \eqref{eq:karamata} with $\eta(t)=o\left(1/\log h(t)\right)$ as $t\to \infty$.
\end{lemma}

\noindent \emph{Proof of Lemma \ref{remarkii}.}
Let us use Lemma \ref{lem:super-slow-rep}. Since $h$ is a slowly varying function and $h(t)\ge 1+h_0$ for all $t$, the function $\phi(t):=\log (h(e^t))$ is such that $1/\phi$ is locally integrable. Moreover we notice that $\phi(t+\delta\phi(t))
  = \log (h(e^t\, h(e^t)^{\delta}))$ and by using Potter's bound \eqref{potter-bound} for $h$ we obtain for all $\epsilon>0$ that there exists $T>0$ such that for $t>T$ we have 
\[
\frac{h(e^t\,h(e^t)^{\delta})}{h(e^t)}\le (1+\epsilon)\, \max \set{h(e^t)^{\delta \epsilon}\, ,\, h(e^t)^{-\delta \epsilon}}\, .
\]
Since $h(t)\ge 1+h_0$, we have $\phi(t)\ge \log(1+h_0)$, hence for all $\epsilon\in (0,h_0)$
\[
 \phi(t+\delta\phi(t)) \le \log (1+\epsilon) + (1+\delta \epsilon)\, \phi(t) \le (2+\delta \epsilon)\, \phi (t)\, .
\]
Analogously, we obtain $\phi(t+\delta\phi(t))\ge \left(\left(1- \delta\epsilon\right)+\log (1-\epsilon)/\log (1+h_0)\right)\phi(t)$. If we choose $\epsilon$ sufficiently small, then the factor before $\phi(t)$ is bounded away from zero which implies that $\phi$ is self-controlled.

Then by Lemma \ref{lem:super-slow-rep}, we obtain that if $L$ is a slowly varying function written as in \eqref{eq:karamata} with $\eta(t) = o(1/\log h(t))$ as $t\to \infty$, it is super-slowly varying with rate function $h$. Let now $\tilde h$ be a function satisfying $c\le \tilde h(n) \le h(n)$ for all $n$ and for a constant $c>0$. Up to dividing by $c$ and using that $\eta(t) = o(1/\log (h(t)/c))$ as $t\to \infty$, we can set $c=1$. Then we can write $\tilde h(n) = h(n)^{\delta(n)}$ where $\delta(n) \in [0,1]$ for all $n$. By Definition \ref{def-ssv} it follows
 \[
\left| \frac{L(n\, \tilde h(n))}{L(n)}-1\right| = \left| \frac{L(n\,  h(n)^{\delta(n)})}{L(n)}-1\right| \le \sup_{\delta\in[0,1]} \left| \frac{L( n\, h(n)^\delta)}{L(n)}-1\right| \to 0 \quad \text{as $n\to \infty$.}  
 \]
and the lemma is proved.
\qed

\end{document}